\documentclass{elsarticle} 

\usepackage{amsmath, amssymb, amsfonts, amsthm}
\usepackage{graphicx}
\usepackage{epstopdf} 

\usepackage{caption}
\usepackage{subcaption} 
\usepackage{float}      

\usepackage{algorithm}
\usepackage{algpseudocode}

\usepackage[colorlinks=true, allcolors=blue, pdftex]{hyperref}

\theoremstyle{plain} 
\newtheorem{theorem}{Theorem}[section]
\newtheorem{lemma}[theorem]{Lemma}

\theoremstyle{definition} 

\theoremstyle{remark} 
\newtheorem{remark}[theorem]{Remark}

\renewcommand{\mathbf}[1]{\boldsymbol{#1}}
\newcommand{\dd}{\mathrm{d}}
\newcommand{\ii}{\mathrm{i}}

\hypersetup{
    pdfauthor={Lifu Song, Tingyue Li, Jin Cheng},
    pdftitle={A Robust Learning-Based Method for the Helmholtz Equation in Dissipative Media and Complex Domains}
}

\begin{document}
\begin{frontmatter}

\title{A Robust Learning-Based Method for the Helmholtz Equation in Dissipative Media and Complex Domains}

\author[a]{Lifu Song}
\author[b]{Tingyue Li}
\author[a]{Jin Cheng\corref{cor1}}
\cortext[cor1]{Corresponding author}
\ead{jcheng@fudan.edu.cn} 

\address[a]{School of Mathematical Sciences, Fudan University, Shanghai 200433, China}
\address[b]{School of Mathematics, Shanghai University of Finance and Economics, Shanghai 200433, China}

\begin{abstract}
To mitigate pollution effects in high-frequency Helmholtz problems, Learning-based Numerical Methods (LbNM) reconstruct solution operators using complete systems of exact solutions. However, the previously used fundamental-solution (FS) basis suffers from instability in dissipative media and requires sensitive geometric tuning. In this paper, we propose a robust alternative using a Bessel basis (BB). From a learning theory perspective, the BB forms a complete hypothesis space of standing waves, ensuring immunity to dissipation-induced signal loss. We establish a convergence result that depends on intrinsic regularity. Numerical experiments demonstrate that the proposed method achieves machine-precision accuracy in dissipative regimes where FS fails, significantly outperforms the Finite Element Method (FEM) in efficiency, and demonstrates the framework's geometric extensibility via a multi-center strategy.
\end{abstract}


\end{frontmatter}

\section{Introduction}
\label{sec:intro}

The Helmholtz equation is a foundational model in wave propagation phenomena, essential to fields ranging from acoustics and electromagnetics to geophysics. A significant computational challenge arises in the high-frequency regime, particularly when the medium exhibits energy dissipation \cite{IsakovLu2018,Wang2024Stability}. Such conditions are prevalent in engineering applications, including ultrasound in biological tissue \cite{Cobbold2006Ultrasound}, acoustic absorbing liners in aeroacoustics \cite{Nayfeh1975Aero}, and seismic wave propagation in viscoelastic media \cite{Carcione2014Seismic}.

Traditional numerical methods, such as the Finite Element Method (FEM), face severe difficulties in this regime due to the pollution effect. To maintain accuracy as the wavenumber $k$ grows, the mesh density must increase super-linearly, leading to prohibitive computational costs \cite{Ainsworth2004Dispersion, Babuskasauter1997Pollution,  IhlenburgBabushka1997HP}. 

To mitigate these issues, the Learning-based Numerical Method (LbNM) \cite{Chen2025} has been proposed. The core idea is to reconstruct the solution operator by learning from a hypothesis space spanned by a set of exact solutions. The performance of the method is thus critically dependent on the quality and robustness of this hypothesis space. In our previous work \cite{Chen2025}, we utilized a basis of fundamental solutions (FS). While fundamental solutions constitute a complete system, they exhibit inherent physical limitations when applied to dissipative media or complex geometries. The fundamental solution $H_0^{(1)}(kr)$ represents an outgoing wave. In a damped medium with a complex wavenumber $\text{Im}(k) > 0$, this kernel decays exponentially with distance. 

\begin{figure}[htbp]
    \centering
    \includegraphics[width=0.6\textwidth]{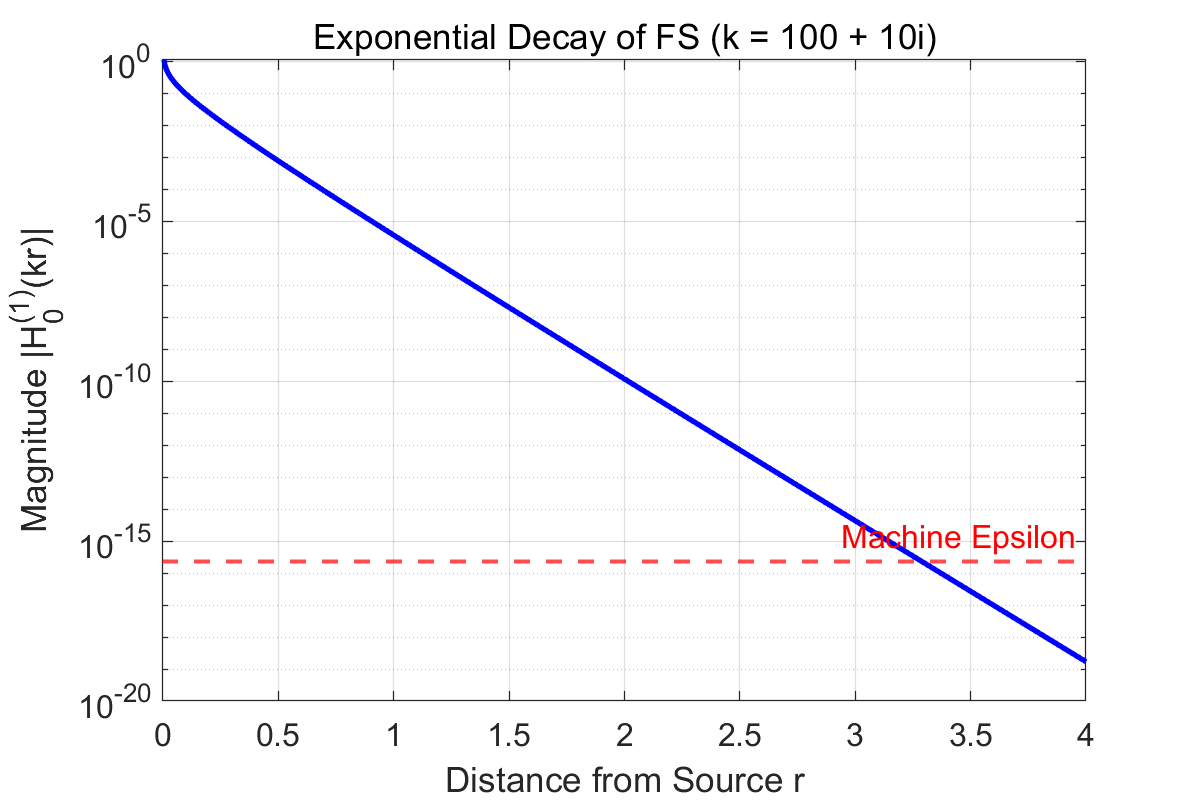}
    \caption{Exponential decay of the fundamental solution $H_0^{(1)}(kr)$ in a dissipative medium ($k=100+10\mathrm{i}$).}
    \label{fig:decay_fs}
\end{figure}

As illustrated in Figure \ref{fig:decay_fs}, for a high-frequency damped wave ($k=100+10\mathrm{i}$), the magnitude of the kernel drops below machine precision within a short distance. This imposes a severe restriction on the source placement. Placing points far from the boundary leads to numerical underflow and ill-conditioning due to the kernel's decay, while placing them close to the boundary results in slow convergence \cite{Chen2025}. Furthermore, for complex engineering geometries, the reliance on a fixed source configuration (typically a circle) creates a geometric mismatch, resulting in an uneven approximation capability across the domain.

This paper addresses these limitations by proposing the Bessel basis (BB) that is intrinsically robust to both physical dissipation and geometric irregularity. Unlike fundamental solutions, Bessel functions represent global standing waves centered at the origin. The advantages of the proposed method are summarized as follows:

\begin{enumerate}
    \item Robustness in Dissipative Media: We demonstrate that the Bessel basis avoids the exponential signal loss inherent to fundamental solutions in damped media. The BB-LbNM maintains high accuracy even in highly dissipative environments where the FS-LbNM fails due to numerical underflow.
    
    \item High-Frequency Efficiency and Accuracy: Compared to the Finite Element Method, the BB-LbNM effectively mitigates the pollution effect. By utilizing exact solutions of the governing equation, the BB-LbNM achieves machine-precision accuracy with significantly lower computational cost, making it ideal for high-frequency simulations.
    
    \item Intrinsic Geometric Adaptivity: The Bessel basis eliminates the need for configuring external source poles. This allows the BB-LbNM to adapt naturally to complex, non-convex geometries where standard FS configurations are inefficient.

    \item Extensibility to Extreme Geometries: We demonstrate that the framework can be extended to non-star-shaped domains where a standard single-center expansion fails. This failure is caused by the drastic difference in distances from the center to the boundary, which makes the basis functions numerically unstable. We show that a multi-center strategy effectively solves this problem.
\end{enumerate}

The remainder of this paper is organized as follows. Section \ref{sec:methodology} details the LbNM framework with the Bessel basis. Section \ref{sec:theory} provides the convergence analysis. Section \ref{sec:numerics} presents comprehensive numerical experiments comparing the proposed method against FS-LbNM and FEM benchmarks, and further explores a multi-center strategy for extreme geometries. Section \ref{sec:conclusion} provides concluding remarks.

\section{The Learning-Based Method with a Bessel basis}
\label{sec:methodology}

We consider the interior Dirichlet problem for the Helmholtz equation in a bounded, piecewise-smooth domain $\Omega \subset \mathbb{R}^2$:
\begin{equation}
\left\{ \begin{array}{ll} 
\Delta u(x) + k^2 u(x) = 0, & x \in \Omega, \\
u(x) = f(x), & x \in \partial\Omega,
\end{array} \right.
\label{eq:helmholtz}
\end{equation}
where $k \in \mathbb{C}$ is the complex wavenumber with $\text{Im}(k) \ge 0$, representing a potentially dissipative medium. We assume $k^2$ is not a Dirichlet eigenvalue of $-\Delta$ on $\Omega$ to ensure well-posedness.

\subsection{The General Learning Algorithm}
The LbNM seeks to construct a discrete approximation of the solution operator that maps the boundary function $f$ to the solution $u(x)$ at any target point $x \in \Omega$. Let $\{x_j\}_{j=1}^N$ be a set of collocation points on the boundary $\partial\Omega$. The boundary data is represented by the vector $\mathbf{f} = (f(x_1), \dots, f(x_N))^\top$. The numerical solution at a point $x$ is then approximated by a linear combination of these boundary values:
\begin{equation*}
    u(x) \approx u^{(c)}(x) = \sum_{j=1}^N a_j(x) f(x_j) = \mathbf{a}(x) \mathbf{f},
\end{equation*}
where $\mathbf{a}(x) = (a_1(x), \dots, a_N(x))$ is the vector of operator coefficients for the target point $x$.

To determine $\mathbf{a}(x)$, we use a set of $M$ known training solutions, $\{u_i\}_{i=1}^M$, which form a basis. For each training solution, we have the approximate relation $u_i(x) \approx \mathbf{a}(x) \mathbf{u}_i$, where $\mathbf{u}_i = (u_i(x_1), \dots, u_i(x_N))^\top$. Stacking these relations forms the linear system:
\begin{equation*}
    \mathbf{b}(x) \approx \mathbf{a}(x) V,
\end{equation*}
where $\mathbf{b}(x) = (u_1(x), \dots, u_M(x))$ is the vector of training solutions evaluated at $x$, and $V$ is the $N \times M$ matrix whose columns are the boundary data of the training solutions, $V_{ji} = u_i(x_j)$.

To find a stable solution for $\mathbf{a}(x)$, we solve the Tikhonov minimization problem:
\begin{equation}
    \mathbf{a}_*(x) = \arg\min_{\mathbf{a} \in \mathbb{C}^{1 \times N}} \left\{ \|\mathbf{a}V - \mathbf{b}(x)\|^2 + \alpha \|\mathbf{a}\|^2 \right\},
    \label{eq:tikhonov}
\end{equation}
where $\alpha > 0$ is a regularization parameter. The unique solution is given by:
\begin{equation}
    \mathbf{a}_*(x) = \mathbf{b}(x) V^* (VV^* + \alpha I)^{-1},
    \label{eq:solution_a}
\end{equation}
where $V^*$ is the conjugate transpose of $V$. The choice of the basis $\{u_i\}$ is the central component of the method.

The inclusion of the regularization term $\alpha \|\mathbf{a}\|^2$ is crucial. Since the basis functions are chosen to form a complete system, the resulting matrix $V$ inevitably becomes ill-conditioned as $M$ increases due to the numerical linear dependence of the basis. Tikhonov regularization mitigates this ill-posedness, allowing for stable reconstruction of the operator even with highly redundant basis sets.

Directly evaluating \eqref{eq:solution_a} requires inverting an $N \times N$ matrix, which has a computational complexity of $O(N^3)$. In typical applications, the number of boundary collocation points $N$ often exceeds the number of basis functions $M$. To improve efficiency, we utilize the matrix identity $V^* (VV^* + \alpha I)^{-1} = (V^* V + \alpha I)^{-1} V^*$. This transforms the solution into:
\begin{equation}
    \mathbf{a}_*(x) = \mathbf{b}(x) (V^* V + \alpha I)^{-1} V^*.
    \label{eq:solution_a_dual}
\end{equation}
This formulation involves inverting an $M \times M$ matrix, reducing the computational complexity to $O(M^3)$.

\begin{remark}
A key property of the LbNM framework is that the approximation error on the training set is monotonically non-increasing with respect to the expansion of the basis set. Mathematically, adding new training solutions expands the hypothesis space, ensuring that the new optimum is at least as good as the previous one. This property is critical for practical applications. It allows us to add compensatory basis functions targeting specific solution properties, such as local singularities, corner behaviors, or auxiliary sources, without degrading the global approximation quality established by the original basis.

Furthermore, the framework admits an efficient online learning algorithm. When a new basis function is added, the kernel matrix in the primal form $(VV^* + \alpha I)$ undergoes a rank-one update. By applying the Sherman-Morrison-Woodbury formula \cite{Golub2013}, the inverse matrix can be updated sequentially from the previous state. This technique reduces the computational cost of adding a new sample from cubic to quadratic $O(N^2)$, making adaptive strategies that iteratively enrich the basis highly efficient.

\end{remark}

\subsection{The Bessel basis}
This work employs a basis of Bessel functions, which operates directly on the original, unscaled domain $\Omega$. We assume $\Omega$ is contained within a disk $O_\rho$ of radius $\rho$. The basis functions are solutions to the Helmholtz equation in polar coordinates $(r, \theta)$:
\begin{equation}
    \tilde{u}_n(x) = J_n(kr)e^{\ii n\theta}, \quad \text{for } n = -M_h, \dots, M_h,
    \label{eq:bb_basis_unnormalized}
\end{equation}
where $M=2M_h+1$ is the total number of basis functions. For improved numerical conditioning, these functions are normalized since $J_n(x)$ decays rapidly as $n$ grows. The final basis functions are:
\begin{equation}
    u_n(x) = \frac{\tilde{u}_n(x)}{\|\tilde{u}_n\|_{L^2(\partial O_{\rho})}} = \frac{\tilde{u}_n(x)}{\sqrt{2\pi \rho} |J_n(k\rho)|}.
    \label{eq:bb_basis}
\end{equation}

\begin{algorithm}[H]
\caption{The LbNM with a Bessel basis}
\label{alg:bb_lbnm}
\begin{algorithmic}[1]
\Require
Problem domain $\Omega$, wavenumber $k$, boundary function $f(x)$.
A set of $P$ target points $\{z_p\}_{p=1}^P \subset \Omega$.
Algorithmic parameters: Number of basis functions $M=2M_h+1$, number of boundary points $N$, regularization parameter $\alpha$.

\Ensure
Numerical solution at target points, $\{u^{(c)}(z_p)\}_{p=1}^P$.

\Statex
\Procedure{LearnOperator}{}
    \State Define $N$ collocation points $\{x_j\}_{j=1}^N$ on $\partial\Omega$.
    \State Construct the $M$ normalized Bessel basis functions $\{u_n(x)\}_{n=-M_h}^{M_h}$ using \eqref{eq:bb_basis}.
    \State Assemble the $N \times M$ boundary data matrix $V$ with entries $V_{jn} = u_n(x_j)$.
    \State Pre-compute the inverse matrix $W = (VV^* + \alpha I)^{-1}$.
    \For{each target point $z_p$ in $\{z_p\}_{p=1}^P$}
        \State Assemble the $1 \times M$ basis vector $\mathbf{b}(z_p) = (u_{-M_h}(z_p), \dots, u_{M_h}(z_p))$.
        \State Compute the operator coefficients: $\mathbf{a}_*(z_p) = \mathbf{b}(z_p) V^* W$.
    \EndFor
    \State \Return The set of operator coefficients $\{\mathbf{a}_*(z_p)\}_{p=1}^P$.
\EndProcedure

\Statex
\Procedure{ApplyOperator}{$\{\mathbf{a}_*(z_p)\}$}
    \State Evaluate the boundary data vector $\mathbf{f} = (f(x_1), \dots, f(x_N))^\top$.
    \For{each target point $z_p$ in $\{z_p\}_{p=1}^P$}
        \State Compute the solution: $u^{(c)}(z_p) = \mathbf{a}_*(z_p) \mathbf{f}$.
    \EndFor
    \State \Return The solution vector $\{u^{(c)}(z_p)\}_{p=1}^P$.
\EndProcedure
\end{algorithmic}
\end{algorithm}


\section{Theoretical Analysis}
\label{sec:theory}

In this section, we provide a rigorous error analysis for the LbNM equipped with the Bessel basis. We begin by the connection between the fundamental solution basis and the Bessel basis through potential theory. Then we provide the necessary preliminaries and finally derive the convergence rates for the proposed method.

\subsection{Two Discretizations of the Single-Layer Potential}

It is illuminating to view the two basis (FS and BB) as different discretization strategies for the single-layer potential representation of the solution. Suppose the solution $u(x)$ can be continued to a domain $O_d$ that contains the problem domain $\Omega$. The solution can be represented as a single-layer potential on the boundary $\partial O_d$:
\begin{equation}
\label{eq:single layer potential expression of the solution}
    u(x) = \frac{\ii}{4} \int_{0}^{2\pi} H_0^{(1)}(k|x - d e^{\ii\phi}|) g(\phi) \dd\phi, \quad \text{for } |x| < d,
\end{equation}
where $g(\phi)$ is the density function.

The Fundamental Solution (FS) basis used in \cite{Chen2025} corresponds to a spatial discretization of \eqref{eq:single layer potential expression of the solution} (e.g., a trapezoidal rule). If we discretize the integration variable $\phi$ at points $\phi_j$, the integral is approximated by a linear combination of kernels $\Phi(x, de^{\ii\phi_j}) = \frac{\ii}{4} H_0^{(1)}(k|x - de^{\ii\phi_j}|)$, which form the FS basis.

The Bessel Function (BB) basis corresponds to a spectral expansion of the integration kernel. Using Graf's addition theorem, the kernel can be expanded in a Fourier-Bessel series for $|x|=r < d$:
\begin{equation*}
    \frac{\ii}{4} H_0^{(1)}(k|x - d e^{\ii\phi}|) = \frac{\ii}{4} \sum_{n \in \mathbb{Z}} H_n^{(1)}(kd) e^{-\ii n\phi} J_n(kr)e^{\ii n\theta}.
\end{equation*}
Substituting this expansion into \eqref{eq:single layer potential expression of the solution} yields:
\begin{equation*}
\begin{split}
    u(x) &= \int_{0}^{2\pi} \left( \frac{\ii}{4} \sum_{n \in \mathbb{Z}} H_n^{(1)}(kd) e^{-\ii n\phi} J_n(kr)e^{\ii n\theta} \right) g(\phi) \dd\phi \\
    &= \sum_{n \in \mathbb{Z}} \left( \frac{\ii}{4} H_n^{(1)}(kd) \int_{0}^{2\pi} g(\phi) e^{-\ii n\phi} \dd\phi \right) J_n(kr)e^{\ii n\theta}.
\end{split}
\end{equation*}
Recognizing the term in parentheses as the coefficient determined by the density $g$, this recovers the BB basis expansion $u(x) = \sum c_n J_n(kr)e^{\ii n\theta}$. Thus, while the FS-LbNM relies on the convergence of quadrature rules which is sensitive to the placement of nodes, the BB-LbNM relies on the spectral convergence of the kernel expansion, which suggests superior performance for smooth solutions.

\subsection{Preliminaries}
We first introduce a fundamental stability estimate for the Helmholtz equation. This result justifies the strategy of the LbNM, which implies that if the numerical solution approximates the boundary data with high accuracy, the error in the interior is automatically controlled.

\begin{lemma}[Interior Stability Estimate]
\label{lemma:interior stability estimate}
\textit{Let $\Omega$ be a bounded domain. If $u \in H^1(\Omega)$ satisfies the Helmholtz equation $\Delta u + k^2 u = 0$ in $\Omega$ and $k^2$ is not a Dirichlet eigenvalue of the Laplacian on $\Omega$, then the the following estimation holds:}
\begin{equation}
\label{eq:interior_stability}
    \|u\|_{L^2(\Omega)} \le C_{stab} \|u\|_{L^2(\partial \Omega)},
\end{equation}
\textit{where the constant $C_{stab}$ depends on the domain $\Omega$ and the distance of $k^2$ to the Dirichlet spectrum \cite{KuttlerSigillito1978Bounding}.}
\end{lemma}

The validity of the Bessel basis approach relies on the approximation properties of the chosen basis functions. In the theory of elliptic partial differential equations, a set of solutions is called a complete system of solutions if any solution in a given domain can be approximated arbitrarily closely by a linear combination of these functions \cite{Bergman1961Integral, Gilbert1974Constructive, Vekua1967New}. For the Helmholtz equation on a disk, the following lemma rigorously establishes the completeness of the Bessel basis.

\begin{lemma}[Completeness of the Bessel Basis]
\label{lemma:bessel fourier series}\textit{Assume $k^2$ is not a Dirichlet eigenvalue of the Laplacian on $O_d$. If $u \in H^1(O_d)$ satisfies the Helmholtz equation $(\Delta + k^2)u = 0$, then $u$ admits a unique series expansion that is convergent in $H^1(O_d)$:}
\begin{equation*}
    u(r, \theta) = \sum_{n \in \mathbb{Z}} c_n J_n(kr) e^{\ii n\theta}.
\end{equation*}
\end{lemma}
\begin{proof}
Since $u \in H^1(O_d)$, by the standard Trace Theorem \cite{Evans2010pde}, there exists a unique boundary trace $g = \text{tr}(u) \in H^{1/2}(\partial O_d) $. Define the Fourier coefficients of the boundary function $g$ on the circle of radius $d$:
\begin{equation}
\label{eq:fourier coefficient of the boundary value}
    \hat{g}_n = \frac{1}{2\pi} \int_{0}^{2\pi} g(d, \theta) e^{-\ii n \theta} \dd \theta, \quad n \in \mathbb{Z}.
\end{equation}
Since $g \in H^{\frac{1}{2}}(\partial O_d)$, it holds
\begin{equation}
    \sum_{n \in \mathbb{Z}} (1+|n|) |\hat{g}_n|^2 < \infty.
    \label{eq:h12_condition}
\end{equation}
We seek a series solution of the form:
\begin{equation}
\label{eq:bessel type series expansion}
    v(r, \theta) = \sum_{n \in \mathbb{Z}} c_n J_n(kr) e^{\ii n \theta}.
\end{equation}
To match the boundary data at $r=d$ between \eqref{eq:fourier coefficient of the boundary value} and \eqref{eq:bessel type series expansion}, the coefficients $c_n$ must satisfy $c_n J_n(kd) = \hat{g}_n$.  If $J_n(kd) = 0$, then $J_n(kr)e^{\ii n\theta}$ would be the eigenvector corresponding to the Dirichlet eigenvalue $k^2$. By the assumption that $k^2$ is not a Dirichlet eigenvalue, we have $J_n(kd) \neq 0$ for all $n \in \mathbb{Z}$. Thus, the coefficients are determined by
\[
c_n = \frac{\hat{g}_n}{J_n(kd)}.\]

Now we prove that $v \in H^1(O_d)$. Let $v_n(r, \theta) = c_n J_n(kr) e^{\ii n \theta}$ be the $n$-th term. Its squared $H^1$-norm is given by:
\begin{equation}
\label{eq:the squared H1 norm of v_n}
    \|v_n\|_{H^1(O_d)}^2 = 2\pi |c_n|^2 \int_0^d r \left( |J_n(kr)|^2 + |k J_n'(kr)|^2 + \frac{n^2}{r^2}|J_n(kr)|^2 \right) \dd r.
\end{equation}
To explicitly estimate the norm, we utilize the following properties of the Bessel function \cite{abramowitz1964handbook}:
\begin{enumerate}
    \item The derivative recurrence relation: $k J_n'(kr) = \frac{n}{r} J_n(kr) - k J_{n+1}(kr)$.
    \item The asymptotic formula for large order $n$: $J_n(z) \sim \frac{1}{\sqrt{n}} \left(\frac{ez}{2n}\right)^n.$
    \item The reflection formula: $J_{-n}(z) = (-1)^nJ_n(z)$.
\end{enumerate}
Substituting this back into \eqref{eq:the squared H1 norm of v_n} and replacing $|c_n|^2$ with $|\hat{g}_n|^2 / |J_n(kd)|^2$, we have
\begin{equation*}
    \|v_n\|_{H^1(O_d)}^2 \lesssim  n^2 |\hat{g}_n|^2 \int_0^d \frac{1}{r} \left(\frac{r}{d}\right)^{2|n|} \dd r = \frac{|n|}{2}|\hat{g}_n|^2.
\end{equation*}
Therefore, the condition $\Vert v \Vert_{H^1(O_d)}^2 = \sum_{n \in \mathbb{Z}} \|v_n\|_{H^1(O_d)}^2 < \infty$ is equivalent to $\sum_{n \in \mathbb{Z}} |n| |\hat{g}_n|^2 < \infty$, which is given by \eqref{eq:h12_condition}.

Since $v \in H^1(O_d)$ is a weak solution to the Helmholtz equation and satisfies $\text{tr}(v) = g$, the difference $w = u - v \in H^1(O_d)$ satisfies $\Delta w + k^2 w = 0$ and $\text{tr}(w) = 0$. By the uniqueness of the Dirichlet problem (since $k^2$ is not an eigenvalue), $w = 0$. Hence, $u = v$ in $O_d$.
\end{proof}

For the solution on a general domain $\Omega$, we rely on the Runge property \cite{Lax1956Stability} to approximate it on a larger disk. Here we cite the quantitative version \cite{RulandSalo2019} appropriate for the Helmholtz equation.

\begin{lemma}[Quantitative Runge Approximation]\label{lemma:quantitative runge approximation}
\textit{Let $\Omega \subset \subset \tilde{\Omega} \subset \subset \Omega_{ext}$ be bounded open Lipschitz sets. Define the spaces of solutions:
$$ S_1 := \{ w \in H^1(\Omega) : (\Delta + k^2)w = 0 \text{ in } \Omega \}, $$
$$ S_2 := \{ h \in H^1(\Omega_{ext}) : (\Delta + k^2)h = 0 \text{ in } \Omega_{ext} \}. $$
Then, for any $\varepsilon \in (0,1)$, there exists a constant $C$, parameters $\theta>0$ and $\beta\ge 1$, such that:
\begin{enumerate}
    \item For any $h \in S_1$, there exists $u \in S_2$  satisfying:
    \begin{equation}
    \label{eq:non-continuable runge approximation}
        \| h - u \|_{L^2(\Omega)} \le \varepsilon \| h \|_{H^1(\Omega)}, \quad \| u \|_{H^{1/2}(\partial \Omega_{ext})} \le C e^{C \varepsilon^{-\theta}} \| h \|_{L^2(\Omega)},\quad .
    \end{equation}
    \item If $h$ can be continued to $\tilde{h} \in H^1(\tilde{\Omega})$ satisfying the Helmholtz equation, then there exists $u \in S_2$  satisfying a polynomial bound:
    \begin{equation}
    \label{eq:continuable runge approximation}
        \| \tilde{h} - u \|_{L^2(\Omega)} \le \varepsilon \| \tilde{h} \|_{H^1(\tilde{\Omega})}, \quad \| u \|_{H^{1/2}(\partial \Omega_{ext})} \le C \varepsilon^{-\beta} \| \tilde{h} \|_{L^2(\Omega)}.
    \end{equation}
\end{enumerate}
}
\end{lemma}
This lemma guaranties that solutions on $\Omega$ can be approximated by solutions on a larger disk $O_d$, allowing us to utilize the convergence result on $O_d$.

\subsection{Convergence Analysis}

We now establish the convergence rate for the case where the solution $u$ to problem \eqref{eq:helmholtz} can be continued to $O_d$, satisfying $\Omega \subset O_\rho \subset O_d$.

Recall that in the BB-LbNM algorithm, we utilize the normalized basis functions on the boundary $\partial O_\rho$:
\begin{equation}
    u_n(x) = \frac{J_n(kr)e^{\ii n\theta}}{ \|J_n(kr)e^{\ii n\theta}\|_{L^2(\partial O_\rho)} } = \frac{J_n(kr)e^{\ii n\theta}}{ \sqrt{2\pi\rho} |J_n(k\rho)| }.
    \label{eq:normalized_basis}
\end{equation}

\begin{theorem}[Bessel Series Approximation] \label{thm:bessel series approximation}
\textit{Assume $\Omega \subset O_\rho \subset O_d$, and $k^2$ is not a Dirichlet eigenvalue of the Laplacian on $O_{\rho}$ and $O_d$. If $u \in H^1(O_d)$ satisfies the Helmholtz equation $(\Delta + k^2)u = 0$,  Then there exists a linear combination $u^{(\mu)}$ of the first $M=2M_h+1$ normalized Bessel basis functions:
\begin{equation*}
    u^{(\mu)}(x) = \sum_{n=-M_h}^{M_h} \mu_{n} u_n(x),
\end{equation*}
such that the approximation error on $O_{\rho}$ is bounded by
\begin{equation}
    \| u - u^{(\mu)} \|_{L^2(O_\rho)} \le C \|u\|_{L^2(\partial O_d)} \left(\frac{\rho}{d}\right)^{M_h}.
    \label{eq:bb_approx_error}
\end{equation}
Furthermore, the coefficient vector $\mathbf{\mu} = \{\mu_n\}$ is bounded by
\begin{equation}
\label{eq:thm:bessel series approx:coeffiecient norm}
    \|\mathbf{\mu}\|_{\ell^2} \le C' \|u\|_{L^2(\partial O_d)}.
\end{equation}
The constants $C$ and $C'$ depend on $k, \rho, d$ but are independent of $M_h$.}
\end{theorem}

\begin{proof}
By Lemma \ref{lemma:bessel fourier series}, $u$ admits a Bessel basis expansion in $O_d$, 
\[u(r, \theta) = \sum_{n \in \mathbb{Z}} c_n J_n(kr) e^{\ii n\theta}.\]
Evaluating this on the boundary $\partial O_d$, we have by Parseval's identity
\begin{equation*}
    \|u\|_{L^2(\partial O_d)}^2 = \sum_{n \in \mathbb{Z}} |c_n|^2 2\pi d |J_n(kd)|^2.
\end{equation*}
Thus for any $n$, \[|c_n|^2 \le \frac{\|u\|_{L^2(\partial O_d)}^2}{2\pi d |J_n(kd)|^2}.\]
The squared error of the truncated series on $\partial O_\rho$ reads
\begin{equation}
\label{eq:thm_bessel_series_approx:error of the truncated series on partial O_rho}
    \|u - u^{(\mu)}\|_{L^2(\partial O_\rho)}^2 = \sum_{|n| > M_h} |c_n|^2 2\pi \rho |J_n(k\rho)|^2.
\end{equation}
Substituting the bound for $|c_n|^2$ and using the asymptotic behavior 
\[|J_n(z)| \sim \frac{1}{\sqrt{ |n|}}(\frac{e|z|}{2|n|})^{|n|},\]
we have the ratio
\begin{equation}
\label{eq:asymptotic ratio of bessel function}
\left| \frac{J_n(k\rho)}{J_n(kd)} \right| \sim \left( \frac{\rho}{d} \right)^{|n|}.
\end{equation}
Therefore, \eqref{eq:thm_bessel_series_approx:error of the truncated series on partial O_rho} is bounded by a geometric series
\begin{equation*}
    \|u - u^{(\mu)}\|_{L^2(\partial O_\rho)}^2 \le C \|u\|_{L^2(\partial O_d)}^2 \sum_{|n| > M_h} \left(\frac{\rho}{d}\right)^{2|n|} \le C' \|u\|_{L^2(\partial O_d)}^2 \left(\frac{\rho}{d}\right)^{2M_h}.
\end{equation*}
By Lemma \ref{lemma:interior stability estimate}, we have 
\begin{equation}
\label{eq:prior estimate for the helmholtz equation}
\|u - u^{(\mu)}\|_{L^2(O_\rho)} \le C_{stab} \|u - u^{(\mu)}\|_{L^2(\partial O_\rho)},
\end{equation}
which yields the error bound.

For the coefficient estimation, since $\{u_n\}$ is orthonormal on $L^2(\partial O_\rho)$, we have $\|\mathbf{\mu}\|_{\ell^2} = \|u^{(\mu)}\|_{L^2(\partial O_\rho)}$. Because $u^{(\mu)}$ is the orthogonal projection, $\|u^{(\mu)}\|_{L^2(\partial O_\rho)} \le \|u\|_{L^2(\partial O_\rho)}$.
Note that by trace theorems and the interior regularity result \cite[Section 6.3]{Evans2010pde}, we have 
\[\|u\|_{L^2(\partial O_{\rho})} \le C \|u\|_{H^1(O_\rho)} \le C' \|u\|_{L^2(\partial O_d)},\]
which yields the final result.
\end{proof}

\begin{theorem}[BB-LbNM Error]\label{thm:BB-LbNM error}
\textit{Under the assumptions of Theorem \ref{thm:bessel series approximation}, let the LbNM algorithm \ref{alg:bb_lbnm} be applied with a regularization parameter $\alpha \sim N(d/\rho)^{-2M_h}$. The resulting numerical solution $u^{(c)}(x) = \mathbf{a}_*(x)\mathbf{f}$ has the error estimate:}
\begin{equation*}
    \| u - u^{(c)} \|_{L^2(\Omega)} \le C \|u\|_{L^2(\partial O_d)} \left( \left(\frac{\rho}{d}\right)^{M_h} + \frac{1}{N} \right).
\end{equation*}
\end{theorem}
\begin{proof}
First, we establish the equivalence between the learned operator action and a coefficient fitting problem. By the explicit formula for the regularized solution \eqref{eq:solution_a}, the numerical solution at $x$ is
\begin{equation*}
    u^{(c)}(x) = \mathbf{a}_*(x) \mathbf{f} = \mathbf{b}(x) V^* (V V^* + \alpha I)^{-1} \mathbf{f}.
\end{equation*}
Using the matrix identity $V^* (V V^* + \alpha I)^{-1} = (V^* V + \alpha I)^{-1} V^*$, this can be rewritten as:
\begin{equation*}
    u^{(c)}(x) = \mathbf{b}(x) \mathbf{c}_*, \quad \text{where} \quad \mathbf{c}_* = (V^* V + \alpha I)^{-1} V^* \mathbf{f}.
\end{equation*}
Notice that $\mathbf{c}_*$ is precisely the unique minimizer of the following functional over coefficients $\mathbf{c} \in \mathbb{C}^M$:
\begin{equation}
    \mathbf{c}_* = \underset{\mathbf{c}\in \mathbb{C}^M}{\arg \min} \left\{ \|\mathbf{f} - V\mathbf{c}\|^2 + \alpha \|\mathbf{c}\|^2 \right\}.
    \label{eq:coef_min}
\end{equation}

Now we utilize Theorem \ref{thm:bessel series approximation}. Taking $\mathbf{\mu}$ as a candidate in the minimization problem \eqref{eq:coef_min}, the optimality of $\mathbf{c}_*$ implies
\begin{equation}
\label{eq:optimality}
    \|\mathbf{f} - V\mathbf{c}_*\|^2 + \alpha \|\mathbf{c}_*\|^2 \le \|\mathbf{f} - V\mathbf{\mu}\|^2 + \alpha \|\mathbf{\mu}\|^2.
\end{equation}
The term 
\[\|\mathbf{f} - V\mathbf{\mu}\|^2 = \sum_{j=1}^N |u(x_j) - u^{(\mu)}(x_j)|^2\] 
represents the interpolation error at the boundary points. Let $e(x) = u(x) - u^{(\mu)}(x)$ be the error function that satisfies $\Delta e + k^2 e = 0$ in $O_\rho$.
In light of the Sobolev embedding theorem and the interior estimation for the Helmholtz equation, combined with \eqref{eq:bb_approx_error} we can derive the following estimate:
\begin{equation*}
     \|e\|_{L^\infty(\partial\Omega)} \le C_{emb} \|e\|_{H^s(\Omega)} \le C \|e\|_{L^2(O_\rho)} \le C \|u\|_{L^2(\partial O_d)} \left(\frac{\rho}{d}\right)^{M_h}.
\end{equation*}
Here $s > 1$ such that $H^s(\Omega)$ is continuously embedded into $L^\infty(\Omega)$.
Summing this squared error over the $N$ collocation points yields:
\begin{equation}
\label{eq:discrete approx error of mu}
    \|\mathbf{f} - V\mathbf{\mu}\|^2 = \sum_{j=1}^N |e(x_j)|^2  \le C N \|u\|_{L^2(\partial O_d)}^2 \left(\frac{\rho}{d}\right)^{2M_h}.
\end{equation}
Substituting \eqref{eq:discrete approx error of mu} together with the bound \eqref{eq:thm:bessel series approx:coeffiecient norm} for $\|\mathbf{\mu}\|$ into \eqref{eq:optimality}, and choosing $\alpha \sim N(\rho/d)^{2M_h}$, we have:
\begin{equation*}
    \|\mathbf{f} - V\mathbf{c}_*\|^2 + \alpha \|\mathbf{c}_*\|^2  \le C N \left(\frac{\rho}{d}\right)^{2M_h} \|u\|_{L^2(\partial O_d)}^2.
\end{equation*}
This inequality provides the necessary controls for both terms:
\begin{equation}
\label{eq:discrete bound for the numercial solution}
    \|\mathbf{f} - V\mathbf{c}_*\|^2 \le C N \left(\frac{\rho}{d}\right)^{2M_h} \|u\|_{L^2(\partial O_d)}^2, \quad \text{and} \quad \|\mathbf{c}_*\|^2  \le C  \|u\|_{L^2(\partial O_d)}^2.
\end{equation}

Now, we can estimate the error of the numerical solution $u^{(c)}(x) = \sum_{n=-M_h}^{M_h} c_{*n} u_n(x)$ on the boundary $\partial\Omega$. We introduce a partition of the boundary $\partial \Omega = \cup_{k=1}^N \Gamma_{k}$ according to the collocation points $\{x_k\}_{k=1}^N$, such that $x_k \in \Gamma_k$ and the arc length satisfies $|\Gamma_k| \le {C}/{N}$.
Denote the total error function by $v(x) = u(x) - u^{(c)}(x)$. Our goal is to estimate $\|v\|_{L^2(\partial\Omega)}^2 = \sum_{k=1}^N \|v\|_{L^2(\Gamma_k)}^2$. On each segment $\Gamma_k$, we use the triangle inequality to split the error:
\begin{equation}
\label{eq:thm:error for case 1_split the error}
    \int_{\Gamma_{k}} |v(x)|^2 \dd x \le 2 \int_{\Gamma_{k}} |v(x) - v(x_k)|^2 \dd x + 2 \int_{\Gamma_{k}} |v(x_k)|^2 \dd x.
\end{equation}
The second term in \eqref{eq:thm:error for case 1_split the error} is simply the discrete residual at the collocation point, bounded by the previous step:
\begin{equation}
\label{eq:thm:error for case 1_the approximation error}
    \sum_{k=1}^N \int_{\Gamma_{k}} |v(x_k)|^2 \dd x = \sum_{k=1}^N |\Gamma_k| |v(x_k)|^2 \le \frac{C}{N} \|\mathbf{f} - V\mathbf{c}_*\|^2 \le C \left(\frac{\rho}{d}\right)^{2M_h} \|u\|_{L^2(\partial O_d)}^2.
\end{equation}
For the first term in \eqref{eq:thm:error for case 1_split the error}, we split it into the oscillation of the exact solution and the numerical solution:
\begin{equation}
\label{eq:thm:error for case 1_split the interpolation error}
        \int_{\Gamma_{k}} |v(x) - v(x_k)|^2 \dd x \le 2 \int_{\Gamma_{k}} |u(x) - u(x_k)|^2 \dd x + 2 \int_{\Gamma_{k}} |u^{(c)}(x) - u^{(c)}(x_k)|^2 \dd x.
\end{equation}
Using the Lagrange mean value theorem, we treat the functions as defined on the curve parameterized by arc length. Let $s(x)$ denote the arc length distance of a point $x \in \Gamma_k$ from the collocation point $x_k$. For any $x \in \Gamma_k$, there exists a point $\xi$ on the curve segment between $x$ and $x_k$ such that:
\begin{equation*}
    |u(x) - u(x_k)| \le \left| \frac{\partial u}{\partial \tau}(\xi) \right| |s(x)| \le |\nabla u(\xi)| |s(x)|,
\end{equation*}
where $\frac{\partial u}{\partial \tau}$ denotes the tangential derivative along the boundary.
Integrating over $\Gamma_k$ using arc length measure:
\begin{equation*}
    \int_{\Gamma_{k}} |u(x) - u(x_k)|^2 \dd x \le \max_{x \in \Gamma_k} |\nabla u(x)|^2 \int_{\Gamma_{k}} |s(x)|^2 \dd x.
\end{equation*}
Since the length of the segment is $|\Gamma_k|$, we have $|s(x)| \le |\Gamma_k|$ when $x \in \Gamma_k$. Thus, combining the interior regularity result \cite[Section 6.3]{Evans2010pde} we have:
\begin{equation}
\label{eq:thm:error for case 1_interpolation error of the true solution}
    \int_{\Gamma_{k}} |u(x) - u(x_k)|^2 \dd x \le |\Gamma_k|^3 \max_{x \in \Gamma_k} |\nabla u(x)|^2 \le \frac{C}{N^3} \| u \|_{L^2(\partial O_d)}.
\end{equation}
Similarly for the numerical solution $u^{(c)}$:
\begin{equation*}
    \int_{\Gamma_{k}} |u^{(c)}(x) - u^{(c)}(x_k)|^2 \dd x \le |\Gamma_k|^3 \max_{x \in \Gamma_k} |\nabla u^{(c)}(x)|^2\le \frac{C}{N^3} \| u^{(c)} \|_{L^2(\partial O_{\rho})}.
\end{equation*}
Crucially, the basis functions $\{u_n\}$ are constructed to be orthonormal on $L^2(\partial O_\rho)$. Therefore, the norm of the $u^{(c)}$ on $\partial O_{\rho}$ is exactly the $\ell^2$-norm of its coefficient vector:
\begin{equation*}
    \|u^{(c)}\|_{L^2(\partial O_\rho)} = \left\| \sum_{n=-M_h}^{M_h} c_{*n} u_n \right\|_{L^2(\partial O_\rho)} = \|\mathbf{c}_*\|_{\ell^2}.
\end{equation*}
Using the bound $\|\mathbf{c}_*\|^2 \le C  \|u\|_{L^2(\partial O_d)}^2$ derived in \eqref{eq:discrete bound for the numercial solution}, we obtain:
\begin{equation}
\label{eq:thm:error for case 1_interpolation error of the num solution}
    \int_{\Gamma_{k}} |u^{(c)}(x) - u^{(c)}(x_k)|^2 \dd x  \le \frac{C}{N^3} \| u \|_{L^2(\partial O_d)}.
\end{equation}
Substituting \eqref{eq:thm:error for case 1_interpolation error of the num solution} and \eqref{eq:thm:error for case 1_interpolation error of the true solution} into \eqref{eq:thm:error for case 1_split the interpolation error}, we get the estimation of the first term in \eqref{eq:thm:error for case 1_split the error}. Summing \eqref{eq:thm:error for case 1_split the error} over all $k=1,\dots,N$ and using \eqref{eq:thm:error for case 1_the approximation error} yields:
\begin{equation*}
    \|u - u^{(c)}\|_{L^2(\partial\Omega)}^2  = \|v\|_{L^2(\partial\Omega)}^2\le C \left( \left(\frac{\rho}{d}\right)^{2M_h} + \frac{1}{N^2} \right) \|u\|_{L^2(\partial O_d)}^2.
\end{equation*}
Taking the square root and we derive the estimation:
\begin{equation*}
    \|u - u^{(c)}\|_{L^2(\partial\Omega)} \le C \|u\|_{L^2(\partial O_d)} \left( \left(\frac{\rho}{d}\right)^{M_h} + \frac{1}{N} \right).
\end{equation*}
By Lemma \ref{lemma:interior stability estimate} we get final result.
\end{proof}

\begin{remark}[General Cases]
\label{rmk:general cases}
For a solution $u$ that cannot be continued to a large disk $O_d$, we rely on Lemma \ref{lemma:quantitative runge approximation} to construct an approximation $u_{\epsilon}$ that can be continued to $O_d$. The error analysis follows the trade-off strategy detailed in \cite[Theorem 4.4 and 4.5]{Chen2025}: we balance the Runge approximation error $\varepsilon$ against the growth of the norm $\|u_{\epsilon}\|$ and the stability of the operator.

Let $\delta = (\rho/d)^{-M_h} + \frac{1}{N}$ denote the baseline convergence rate derived in Theorem \ref{thm:BB-LbNM error}. The convergence results for the general cases are:

\begin{enumerate}
    \item  If $u$ can be continued to a domain $\tilde{\Omega}$ slightly larger than $\Omega$ but smaller than $O_d$, the polynomial growth of the Runge approximation \eqref{eq:continuable runge approximation} leads to a Hölder-type error estimate:
    \begin{equation}
        \| u - u^{(c)} \|_{L^2(\partial\Omega)} \le C \|u\|_{H^1(\tilde{\Omega})} \cdot \delta^{\frac{1}{1+\beta}}.
    \end{equation}
    
    \item  If $u$ cannot be continued beyond $\Omega$, the exponential growth of the Runge approximation \eqref{eq:non-continuable runge approximation} leads to a Logarithmic error estimate:
    \begin{equation}
        \| u - u^{(c)} \|_{L^2(\partial\Omega)} \le C \|u\|_{H^1(\Omega)} \cdot \frac{1}{|\ln \delta|^{\frac{1}{\theta}}}.
    \end{equation}
\end{enumerate}
These results confirm that the method converges for all Helmholtz solutions, with the rate determined by the distance to the nearest singularity.
\end{remark}


\section{Numerical Validation}
\label{sec:numerics}

In this section, we present numerical experiments to validate our theoretical claims and demonstrate the practical advantages of the Bessel basis (BB) approach. We perform a direct comparison of the proposed BB-LbNM against the FS-LbNM from our prior work \cite{Chen2025} and the benchmark FEM method. All computations are performed in MATLAB R2023b with an AMD Ryzen 7 7800X3D CPU and 32 GB of RAM.

\subsection{General Experimental Setup}

The objective is to provide a comprehensive comparison of the proposed method against the FS-LbNM and FEM benchmarks, focusing on their performance in dissipative media and complex geometries.

\paragraph{Governing Equation and Physics} We consider the Damped Helmholtz equation to model wave propagation in dissipative media, such as biological tissue or acoustic absorbing liners:
\begin{equation}
    \Delta u + k^2 u = 0, \quad \text{in } \Omega,
    \label{eq:damped_helmholtz}
\end{equation}
where $k = k_r + \mathrm{i}\sigma$ is the complex wavenumber. For the standard benchmarks Test 1 and Test 2, we fix the real part at $k_r = 184.79$ (corresponding to the sound wave in air with a  frequency of 10000 Hz) and vary the damping coefficient $\sigma \ge 0$ to simulate different physical regimes, ranging from lossless $\sigma=0$ to highly dissipative $\sigma \approx 0.2 k_r$. Specific parameters for the exploratory Test 3 will be defined locally in that subsection.

\paragraph{Exact Solutions and Error Metric} To rigorously evaluate the solvers in the comparative benchmarks, we verify the methods against two distinct types of exact solutions representing different physical scenarios:
\begin{itemize}
    \item Case A: Damped Plane Wave.
    \begin{equation*}
        u(x) = e^{\mathrm{i} k x_1} =  e^{\mathrm{i}k_r x_1} e^{-\sigma x_1}.
    \end{equation*}
    This represents a global, smooth wave propagating and decaying in the $x_1$-direction. It is primarily used to test the method's stability against dissipation and its convergence rate for smooth fields.
    
    \item Case B: Damped Dipole Source.
    \begin{equation*}
        u(x) = \Phi_k(\hat{y}_1, x) - \Phi_k(\hat{y}_2, x), \quad \text{where } \Phi_k(y, x) = \frac{\mathrm{i}}{4} H_0^{(1)}(k|x-y|).
    \end{equation*}
    Here, $\Phi_k$ is the fundamental solution with complex wavenumber $k$, and $\hat{y}_{1},\hat{y}_{2}$ are source points located outside the domain. This solution introduces strong gradients and singularities near the boundary, testing the solvers' robustness in reconstructing near-field radiation patterns.
\end{itemize}
The accuracy is quantified using the relative discrete $L^2$ error calculated over a dense grid of interior points with a mesh size of $8 \times 10^{-3}$.

\paragraph{Solver Regularization Strategies} A critical difference between the methods lies in how the regularization parameter $\alpha$ is chosen:
\begin{itemize}
    \item FS-LbNM: Uses the theoretical rule derived in \cite[Theorem 4.2]{Chen2025}, $\alpha \approx MNR^{-2M}$. This represents a carefully optimized baseline.
    \item BB-LbNM: Since the theoretical optimal regularization parameter depends on the radius of analytic continuation $d$, which is typically unknown in practical applications, we employ the Generalized Cross-Validation (GCV) method \cite{Hansen2007} to automatically select $\alpha$. This tests the method's ability to adapt to unknown solution properties without manual tuning.
\end{itemize}

\subsection{Test 1: Stability and Efficiency in Dissipative Media}

\paragraph{Objective} This test compares the stability and efficiency of the BB-LbNM. First, we investigate stability as the physical damping of the medium increases. Theoretically, the fundamental solution $H_0^{(1)}(k r)$ decays exponentially as $e^{-\sigma r}$, which leads to numerical underflow and information loss for the FS-LbNM in high-damping regimes. We verify whether the global nature of the Bessel basis, combined with boundary normalization, overcomes this issue. Second, we compare the BB-LbNM against the FEM to evaluate the computational cost required to achieve high accuracy in the high-frequency regime.

\paragraph{Configuration} We use the rounded kite-shape and flower-shape domains from \cite{Chen2025}. Their parametric equations are detailed in Table \ref{tab:domains}, and the computational setup is illustrated in Figure \ref{fig:setup}.
For the FS-LbNM, the parameters are fixed to the optimized values from \cite{Chen2025}: $N=M=408, R=1.05$ for the kite; $N=M=288, R=1.07$ for the flower. The BB-LbNM uses the same $N$ and $M$ for a fair comparison. The exact solution is Case A (Damped Plane Wave).

\begin{table}[H]
    \centering
    \caption{Parametric equations for the computational domains in Test 1.}
    \label{tab:domains}
    \begin{tabular}{l|l}
        \hline
        \textbf{Shape} & \textbf{Parameter Equation} \\
        \hline
        Rounded kite-shape & 
        \begin{tabular}{@{}l@{}}
            $x=0.5 \cos t+0.3 \cos(2t)-0.2,$ \\
            $y=0.6 \sin t,$ \\
            $t \in [0,2\pi]$
        \end{tabular} \\
        \hline
        Flower-shape & 
        \begin{tabular}{@{}l@{}}
            $x=0.5 \cos t-0.1 \cos(6t)\cos t,$ \\
            $y=0.5 \sin t-0.1 \cos(6t)\sin t,$ \\
            $t \in [0,2\pi]$
        \end{tabular} \\
        \hline
    \end{tabular}
\end{table}

\begin{figure}[htbp]
    \centering
    \begin{subfigure}[b]{0.48\textwidth}
        \centering
        \includegraphics[width=\textwidth]{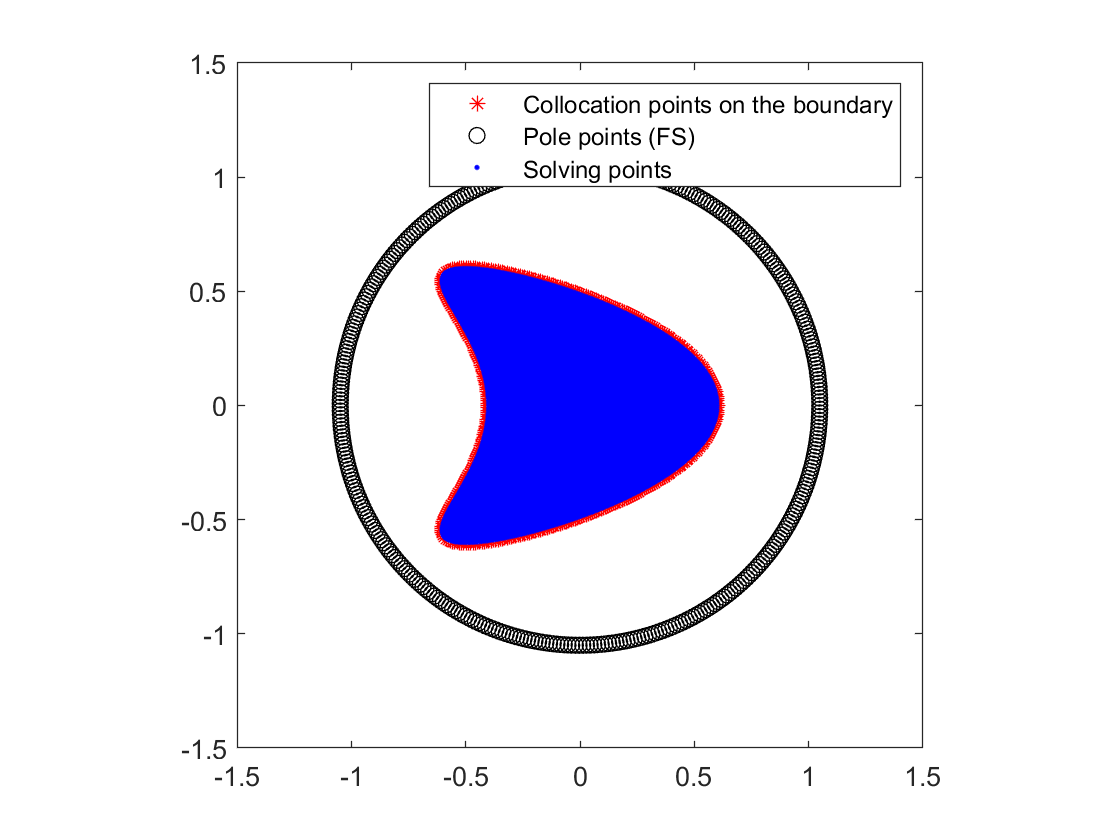}
        \caption{Rounded kite-shape}
        \label{fig:setup_kite}
    \end{subfigure}
    \hfill
    \begin{subfigure}[b]{0.48\textwidth}
        \centering
        \includegraphics[width=\textwidth]{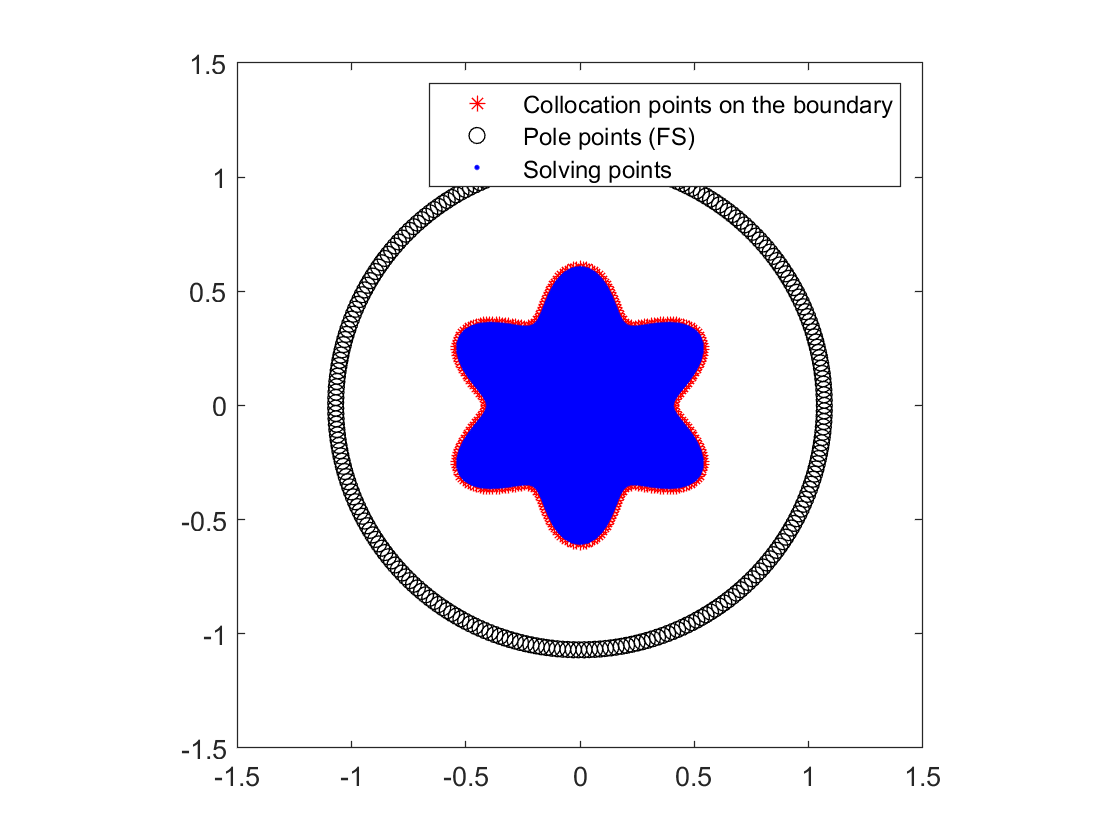}
        \caption{Flower-shape}
        \label{fig:setup_flower}
    \end{subfigure}
    \caption{Illustration of the computational setup for Test 1. The figures show the interior solving points (blue dots), the collocation points on the boundary (red stars), and the exterior pole points used by the benchmark FS-LbNM (black circles).}
    \label{fig:setup}
\end{figure}

\subsubsection{Stability against Damping Variation}
\label{sec:Stability against Damping Variation}
We vary the damping ratio $\sigma/k_r$ from $0\%$ to $20\%$. Table \ref{tab:stability_results} summarizes the results.

\begin{table}[H]
    \centering
    \caption{Stability test under varying damping levels.}
    \label{tab:stability_results}
    \begin{tabular}{l|cc|cc}
        \hline
        & \multicolumn{2}{c|}{\textbf{Kite Geometry}} & \multicolumn{2}{c}{\textbf{Flower Geometry}} \\
        \textbf{Damping} & \textbf{FS Error} & \textbf{BB Error} & \textbf{FS Error} & \textbf{BB Error} \\
        \hline
        0\%    & $1.61 \times 10^{-7}$ & $1.06 \times 10^{-12}$ & $4.09 \times 10^{-6}$ & $3.62 \times 10^{-12}$ \\
        1\%              & $1.28 \times 10^{-7}$ & $1.03 \times 10^{-11}$ & $1.66 \times 10^{-5}$ & $5.64 \times 10^{-12}$ \\
        5\%              & $1.21 \times 10^{-6}$ & $1.21 \times 10^{-11}$ & $2.66 \times 10^{-5}$ & $5.17 \times 10^{-12}$ \\
        20\%  & $3.10 \times 10^{-1}$ & $2.80 \times 10^{-11}$ & $1.00 \times 10^{+0}$ & $1.98 \times 10^{-11}$ \\
        \hline
    \end{tabular}
\end{table}

The numerical results confirm the instability mechanism discussed previously. As shown in Figure \ref{fig:damped_stability}, the error of the FS-LbNM increases by six orders of magnitude as the damping coefficient grows. This supports our analysis that the exponential decay of the fundamental solution ($e^{-\sigma r}$) leads to numerical underflow for source points far from the boundary, resulting in signal loss.

In comparison, the BB-LbNM maintains accuracy ($< 10^{-11}$) for all damping levels. This verifies that the Bessel basis does not suffer from damping. Consequently, the method preserves physical information within the domain and remains stable in dissipative media.

\begin{figure}[htbp]
    \centering
    \begin{subfigure}[b]{0.48\textwidth}
        \centering
        \includegraphics[width=\textwidth]{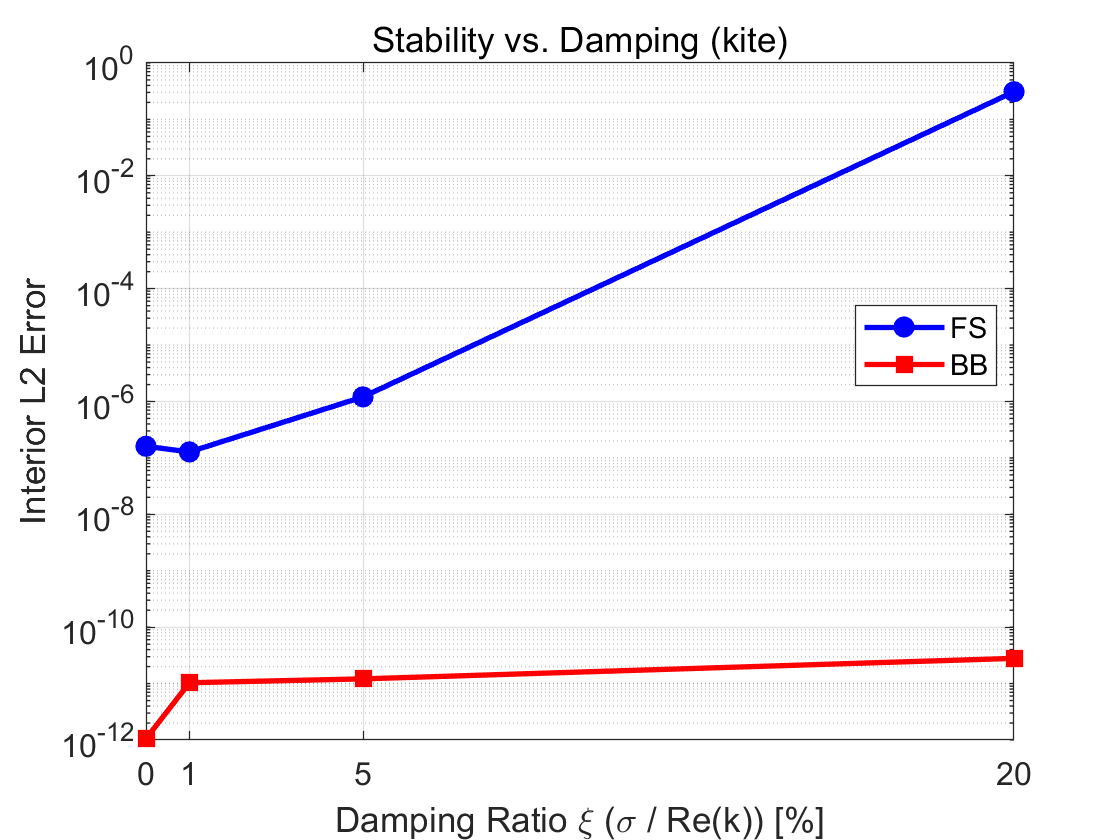}
        \caption{Kite Geometry}
    \end{subfigure}
    \hfill
    \begin{subfigure}[b]{0.48\textwidth}
        \centering
        \includegraphics[width=\textwidth]{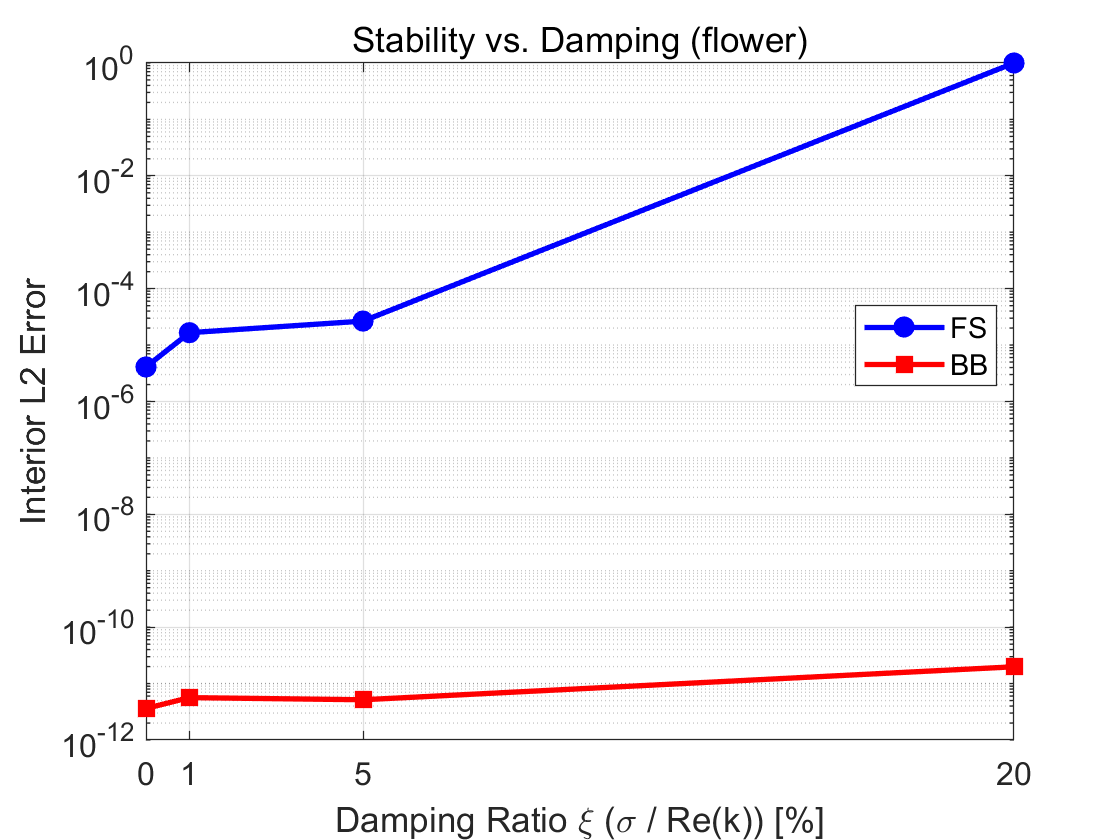}
        \caption{Flower Geometry}
    \end{subfigure}
    \caption{Error vs. Damping Coefficient ($\sigma$). The BB-LbNM (Red) remains stable regardless of dissipation, while the FS-LbNM (Blue) degrades rapidly in high-loss media.}
    \label{fig:damped_stability}
\end{figure}

\subsubsection{Convergence Rate}
The experimental setup follows the Kite geometry case from Section \ref{sec:Stability against Damping Variation}. To analyze the convergence behavior, we fix $N=1000$ and vary $M$ from 180 to 300 under both lossless ($0\%$) and damped ($5\%$) conditions. Figure \ref{fig:convergence_rate} plots the error decay.

A linear fit to the semi-logarithmic data reveals the following convergence slopes:
\begin{itemize}
    \item Lossless (0\%): FS slope $-0.0816$ vs. BB slope $-0.1630$.
    \item Damped (5\%): FS slope $-0.0331$ vs. BB slope $-0.1730$.
\end{itemize}
The BB-LbNM exhibits significantly faster convergence in both cases. Notably, the damping degrades the convergence of the FS method, while the BB method's convergence rate remains robust, demonstrating its superior approximation power in dissipative media.

\begin{figure}[htbp]
    \centering
    \includegraphics[width=0.6\textwidth]{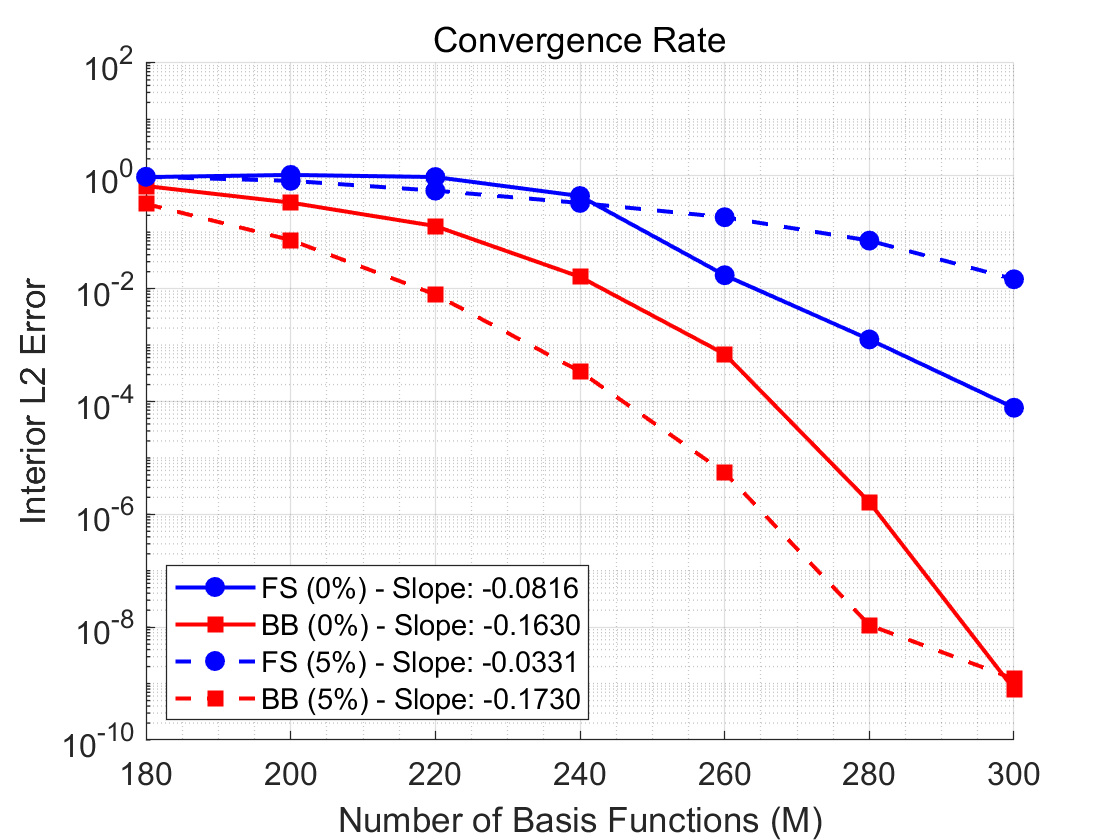}
    \caption{Convergence rate comparison for rounded-kite domain.}
    \label{fig:convergence_rate}
\end{figure}

\subsubsection{Computational Efficiency vs. FEM}
We compare the computational cost of the BB-LbNM against a standard FEM solver. The physical configuration, including the Kite geometry, the Case A exact solution, and the damping coefficient $\sigma \approx 5\%$, remains identical to the setup in Section \ref{sec:Stability against Damping Variation}.

To generate the efficiency curves, we vary the discretization levels for both methods. For the BB-LbNM, we fix the number of boundary collocation points at $N=600$ and increase the number of basis functions $M$ from 180 to 340. For the FEM, we refine the mesh by varying the maximum element size $H_{max}$ through the sequence $\{0.02, 0.01, 0.005, 0.0025, 0.0015\}$, which corresponds to a progressively higher number of degrees of freedom.

For the LbNM, we report two distinct time metrics:
\begin{itemize}
    \item Offline LbNM: This phase encompasses all pre-computations required to construct the solution operator. It includes evaluating the Bessel functions at both the boundary collocation points and the dense grid of internal target points, performing Generalized Cross-Validation to select the regularization parameter, and assembling the inverse operator. It is worth noting that evaluating the special functions at the massive number of internal points constitutes the dominant portion of this computational cost. However, since the Bessel basis functions are defined globally on the ambient disk, these internal evaluations are independent of the specific domain boundary. Consequently, they can be pre-computed and stored, allowing for significant acceleration in subsequent computations even if the domain geometry changes.
    \item Online LbNM: This measures the cost to solve for a new boundary condition once the operator is constructed. This step consists solely of a matrix-vector multiplication.
\end{itemize}

Figure \ref{fig:efficiency} plots the Error vs. CPU Time.

\begin{figure}[htbp]
    \centering
    \includegraphics[width=0.6\textwidth]{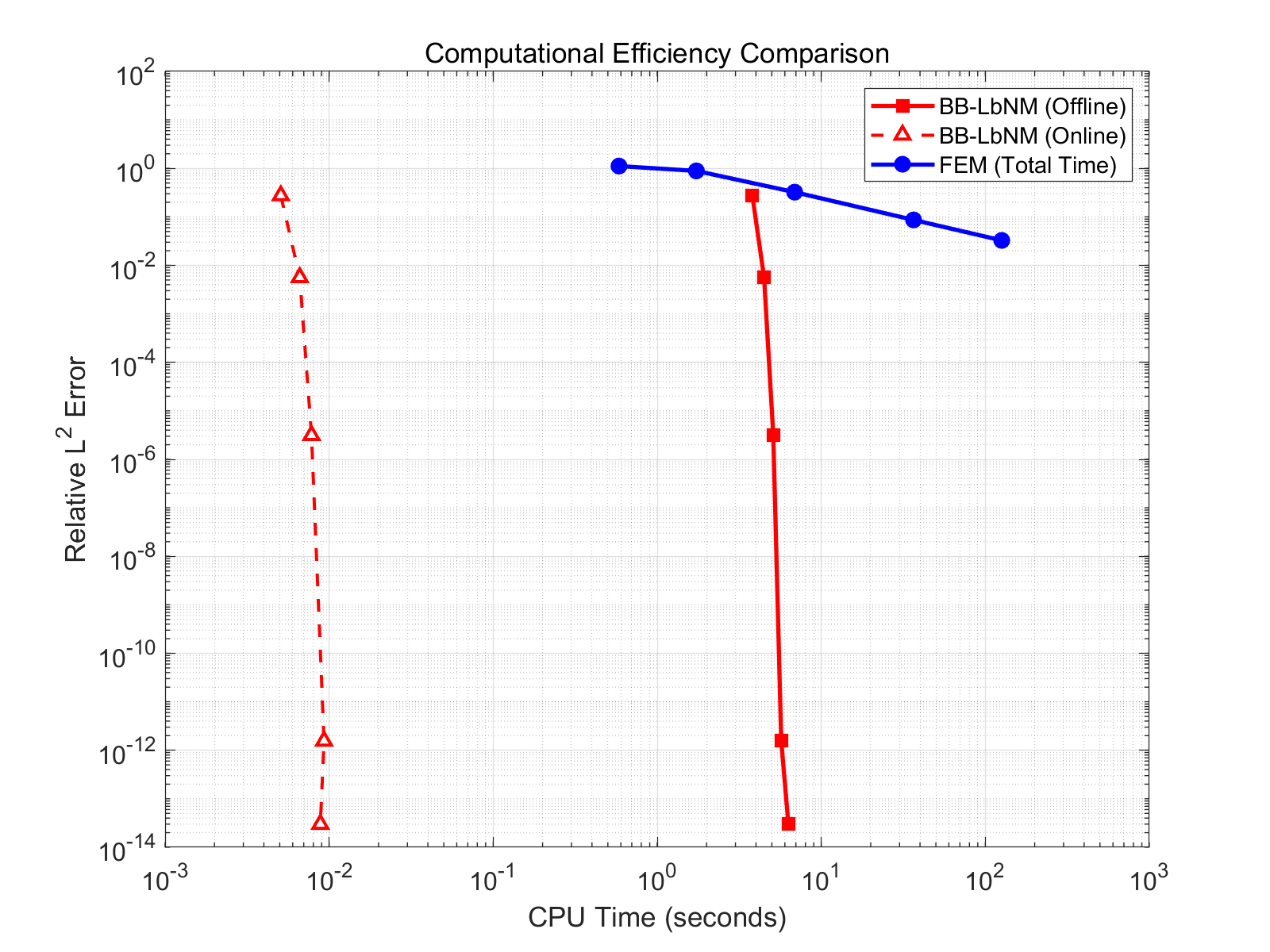}
    \caption{Computational Efficiency Comparison. The BB-LbNM (Red) achieves machine precision ($10^{-14}$) in seconds, whereas the FEM (Blue) struggles to reach $10^{-2}$ even after 120 seconds. The dashed line shows the extremely fast Online evaluation time for BB-LbNM.}
    \label{fig:efficiency}
\end{figure}

The results highlight the dramatic efficiency advantage of the BB-LbNM over the FEM in high-frequency regime. With $M=300$ basis functions, the BB-LbNM achieves an error of $1.56 \times 10^{-12}$ in just 5.7 seconds. In contrast, the FEM requires a very fine mesh ($H_{max}=0.0015$, corresponding to approximately $4.5 \times 10^5$ degrees of freedom) and 126 seconds of computation time to reach an error of only $3.27 \times 10^{-2}$.

Furthermore, the method exhibits exceptional speed in the online phase. Once the solution operator is learned, the BB-LbNM can solve for a new boundary condition in approximately $0.009$ seconds, as indicated by the dashed line in Figure \ref{fig:efficiency}. This capability makes the method ideal for real-time applications or inverse problems requiring repeated evaluations.

\subsection{Test 2: Robustness on Complex Geometries}

\paragraph{Objective} We evaluate the generalizability of the method to complex, engineering-relevant geometries: a Star, a Cabin, and a Plane. We benchmark the proposed BB-LbNM against both the FS-LbNM and the standard FEM. Specifically, we aim to demonstrate: 
\begin{enumerate}
\item The adaptability of the Bessel basis compared to the fixed source configuration of the FS-LbNM.
\item The superior accuracy of the spectral approach compared to the FEM, which suffers from significant pollution errors in the high-frequency regime.
\end{enumerate}

\paragraph{Configuration} The medium is set to $\sigma = 10$ (approx. 5\% damping). We test both exact solutions: Case A (Plane Wave) and Case B (Dipole). The complex geometries and the dipole location for the are shown in Figure \ref{fig:geo_setup}. We fix $N=M=400$ for all cases. The FEM mesh size is taken as $H_{max}=0.001$.

\begin{figure}[htbp]
    \centering
    \includegraphics[width=\textwidth]{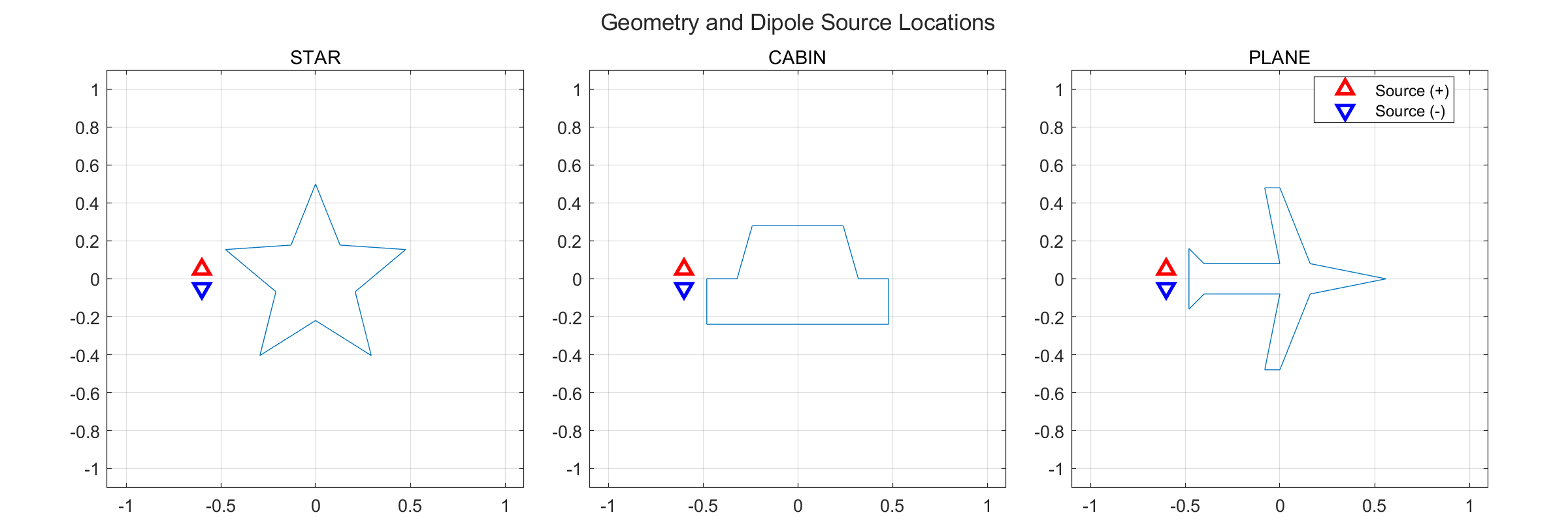}
    \caption{Complex geometries for Test 2: Star, Cabin, and Plane.}
    \label{fig:geo_setup}
\end{figure}

\paragraph{Results} Table \ref{tab:complex_results} summarizes the performance.

\begin{table}[H]
    \centering
    \caption{Robustness test on complex geometries. Comparison of BB-LbNM, FEM, and FS-LbNM.}
    \label{tab:complex_results}
    \begin{tabular}{llcc}
        \hline
        \textbf{Scenario} & \textbf{Method} & \textbf{$L^2$ Error} & \textbf{Time (s)} \\
        \hline
        Star (PlaneWave) 
            & BB-LbNM & $4.51 \times 10^{-14}$ & 0.87 \\
            & FEM     & $1.24 \times 10^{-2}$  & 16.3 \\
            & FS-LbNM & $2.10 \times 10^{-4}$  & --   \\
        \cline{2-4} 
        Star (Dipole1)   
            & BB-LbNM & $2.58 \times 10^{-13}$ & 0.92 \\
            & FEM     & $1.21 \times 10^{-2}$  & 16.4 \\
            & FS-LbNM & $2.47 \times 10^{-3}$  & --   \\
        \hline
        Cabin (PlaneWave)
            & BB-LbNM & $1.74 \times 10^{-14}$ & 1.15 \\
            & FEM     & $1.35 \times 10^{-2}$  & 19.6 \\
            & FS-LbNM & $9.25 \times 10^{-5}$  & --   \\
        \cline{2-4}
        Cabin (Dipole1)  
            & BB-LbNM & $1.34 \times 10^{-12}$ & 1.17 \\
            & FEM     & $1.37 \times 10^{-2}$  & 19.2 \\
            & FS-LbNM & $8.53 \times 10^{-4}$  & --   \\
        \hline
        Plane (PlaneWave)
            & BB-LbNM & $5.24 \times 10^{-13}$ & 0.79 \\
            & FEM     & $1.22 \times 10^{-2}$  & 12.1 \\
            & FS-LbNM & $1.93 \times 10^{-4}$  & --   \\
        \cline{2-4}
        Plane (Dipole1)  
            & BB-LbNM & $4.31 \times 10^{-10}$ & 0.79 \\
            & FEM     & $1.24 \times 10^{-2}$  & 11.9 \\
            & FS-LbNM & $1.51 \times 10^{-3}$  & --   \\
        \hline
    \end{tabular}
\end{table}

 The BB-LbNM demonstrates remarkable robustness, achieving errors between $10^{-10}$ and $10^{-14}$ across all geometries. In contrast, the FS-LbNM error remains limited to the range of $10^{-3} \sim 10^{-4}$.

This performance gap highlights the ``geometric mismatch" issue inherent to the FS method. For irregular geometries such as the Star and the Plane, the fixed circular source configuration is geometrically inefficient. The radial distance from the boundary to the source points varies drastically. This spatial non-uniformity prevents the fixed circular source configuration from providing a balanced approximation across the entire boundary.

Furthermore, the standard FEM achieves only moderate accuracy ($10^{-2}$), which is significantly worse than both LbNM approaches. This degradation is characteristic of the pollution effect (dispersion error) inherent to grid-based methods in the high frequency regime. A fundamental advantage of the LbNM framework is that the basis functions satisfy the governing equation exactly. Consequently, the numerical solution is free from interior dispersion errors. While both LbNM methods surpass the FEM, the BB-LbNM further excels by utilizing a global basis intrinsic to the coordinate system, allowing it to naturally adapt to complex shapes without user intervention.

\subsection{Test 3: Geometric Adaptivity via Multi-Center Expansion}

In the final experiment, we investigate the performance of the BB-LbNM on a non-star-shaped geometry with significant radial disparity. This test evaluates whether the method can maintain stability when the domain geometry imposes severe constraints on a single-center expansion, and demonstrates a multi-center strategy to overcome these limitations.

\paragraph{The Geometric Challenge}
We consider a C-shaped domain, defined by a centerline radius $R_c=1.0$ and width $w=0.4$, which wraps around the origin. The geometric setup is shown in Figure \ref{fig:multicenter_geo_only}.

\begin{figure}[htbp]
    \centering
    \includegraphics[width=0.6\textwidth]{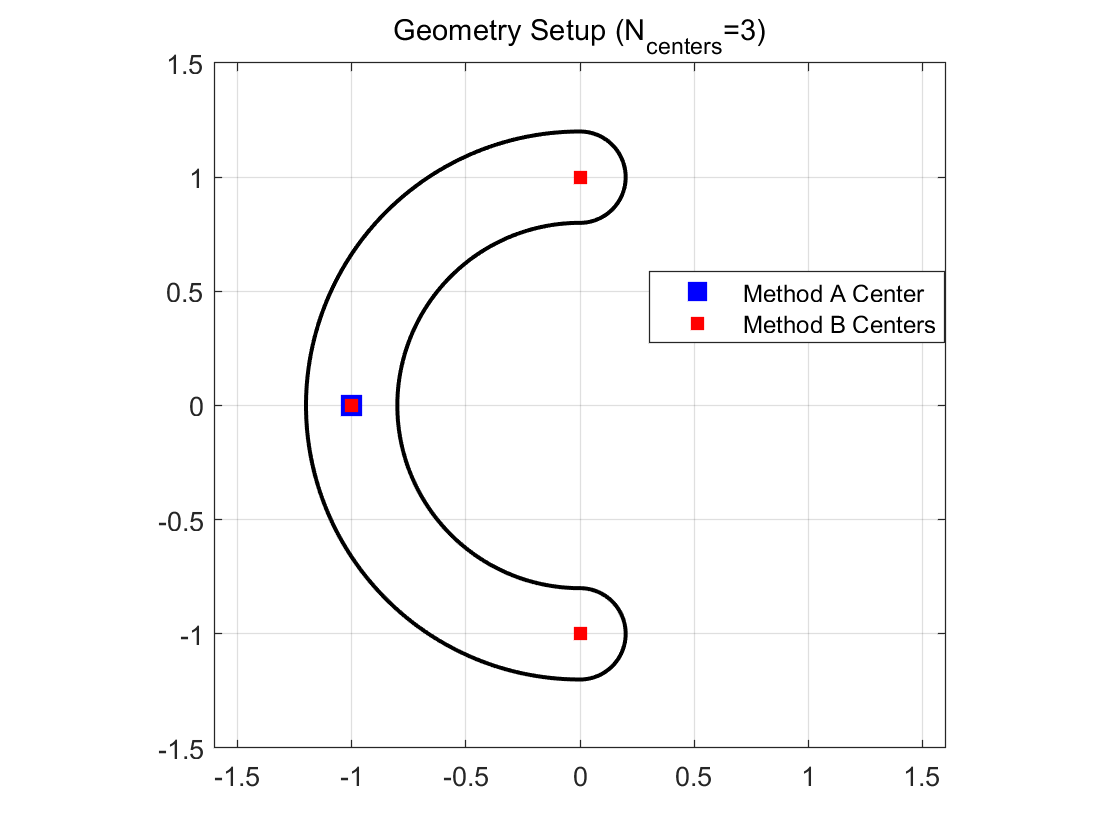}
    \caption{Geometry for Test 3. The blue square represents the single center (Method A), while red squares indicate the three centers used for the Multi-Center strategy (Method B).}
    \label{fig:multicenter_geo_only}
\end{figure}

Similar to the analysis of the Rayleigh hypothesis \cite{Millar1973Rayleigh}, a single-center expansion converges only within the maximal disk extending to the nearest singularity of the solution's analytic continuation. For this non-convex geometry, the singularity may locate within the convex hull limits the convergence radius of an interior center, leaving the domain tips outside the effective convergence region.

Even in the absence of singularities, the significant difference in radial distances creates a barrier for simultaneous approximation. For a single center placed at $(-1, 0)$, the magnitude of high-order Bessel functions at the "back" of the shape ($r_{min} \approx 0.2$) is drastically smaller than at the "tips" ($r_{max} \approx 1.6$). Using the asymptotic ratio derived in Section \ref{sec:theory}, we have:
\begin{equation}
    \left| \frac{J_n(k r_{min})}{J_n(k r_{max})} \right|  \approx \frac{1}{8^{|n|}}.
\end{equation}
For large $n$, this ratio drops far below machine precision. Consequently, the high-order basis functions required to resolve the wave field at the tips are numerically negligible at the back, making it impossible to accurately approximate the boundary conditions at both locations simultaneously.

\paragraph{Strategy}
To resolve this, we adopt a Multi-Center Strategy. This approach mimics the Generalized Multipole Technique (GMT) \cite{Hafner1990GMT}, utilizing local basis functions to target specific geometric features. By distributing expansion centers, we decompose the global approximation problem into local sub-problems, significantly reducing the required order of Bessel functions for each center and avoiding numerical underflow.

To ensure a rigorous comparison of efficiency, we fix the Total Degrees of Freedom (DOF) at $M_{total}=600$ for both methods:
\begin{itemize}
    \item Method A (Single-Center): A single expansion center at $\mathbf{c}_{SC} = (-1, 0)$ with $M=600$ basis functions.
    \item Method B (Multi-Center): Three centers located at $\mathbf{c}_1=(-1, 0)$, $\mathbf{c}_2=(0, 1)$, and $\mathbf{c}_3=(0, -1)$. The DOF is split evenly ($M=200$ per center).
\end{itemize}

\paragraph{Configuration}
We test with a mildly-damped wavenumber $k=50+2i$. Two different boundary conditions (BCs) are applied, as visualized in Figure \ref{fig:multicenter_bcs}:
\begin{enumerate}
    \item External Pulse:
    \begin{equation}
    \begin{split}
        f(x,y) &= \exp\left(-20\left((x - 0.2)^2 + (y - 1)^2\right)\right) \\
               &\quad + \exp\left(-20\left((x + 1.2)^2 + y^2\right)\right).
    \end{split}
    \end{equation}
    This simulates localized excitations near and away from the single center, forcing the basis to capture features at different radial distances simultaneously. 
    \item Constant: $f(x) = 1$. It forces the solver to synthesize a constant boundary profile through highly oscillatory basis functions, serving as a test of the basis's global approximation capability.
\end{enumerate}

\begin{figure}[htbp]
    \centering
    \begin{subfigure}[b]{0.48\textwidth}
        \centering
        \includegraphics[width=\textwidth]{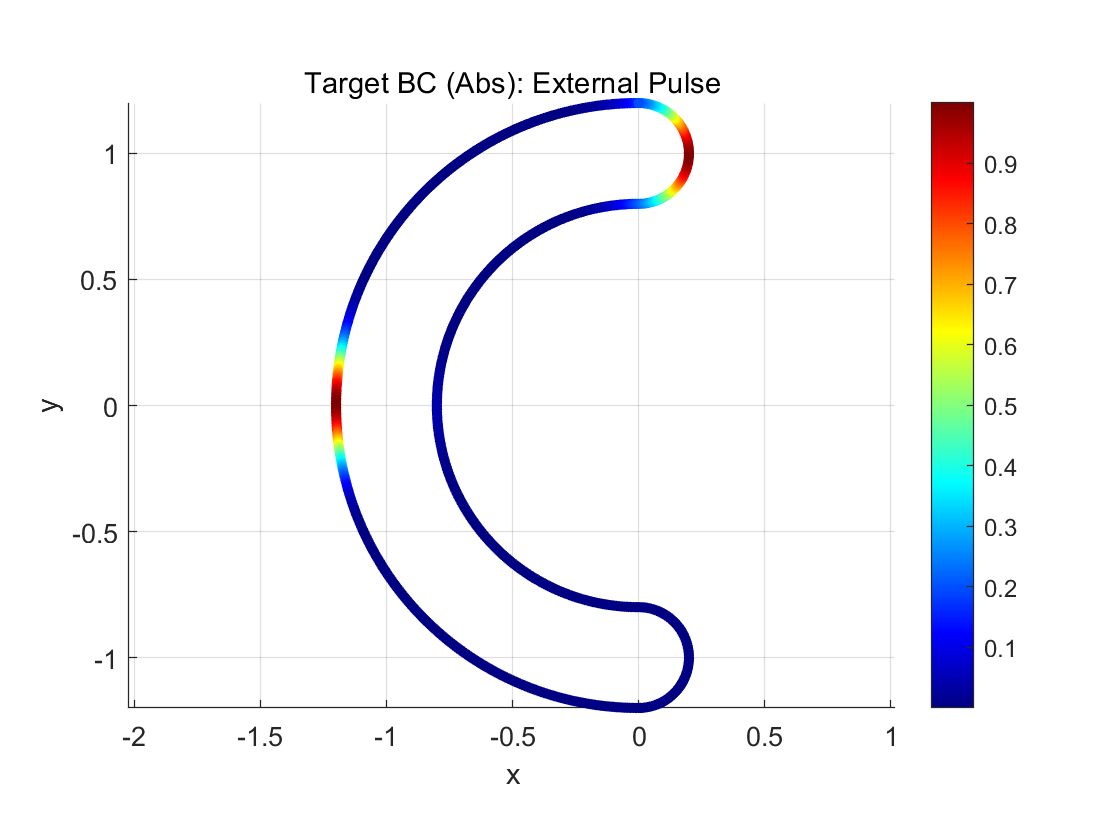}
        \caption{BC: External Pulse}
        \label{fig:bc_pulse}
    \end{subfigure}
    \hfill
    \begin{subfigure}[b]{0.48\textwidth}
        \centering
        \includegraphics[width=\textwidth]{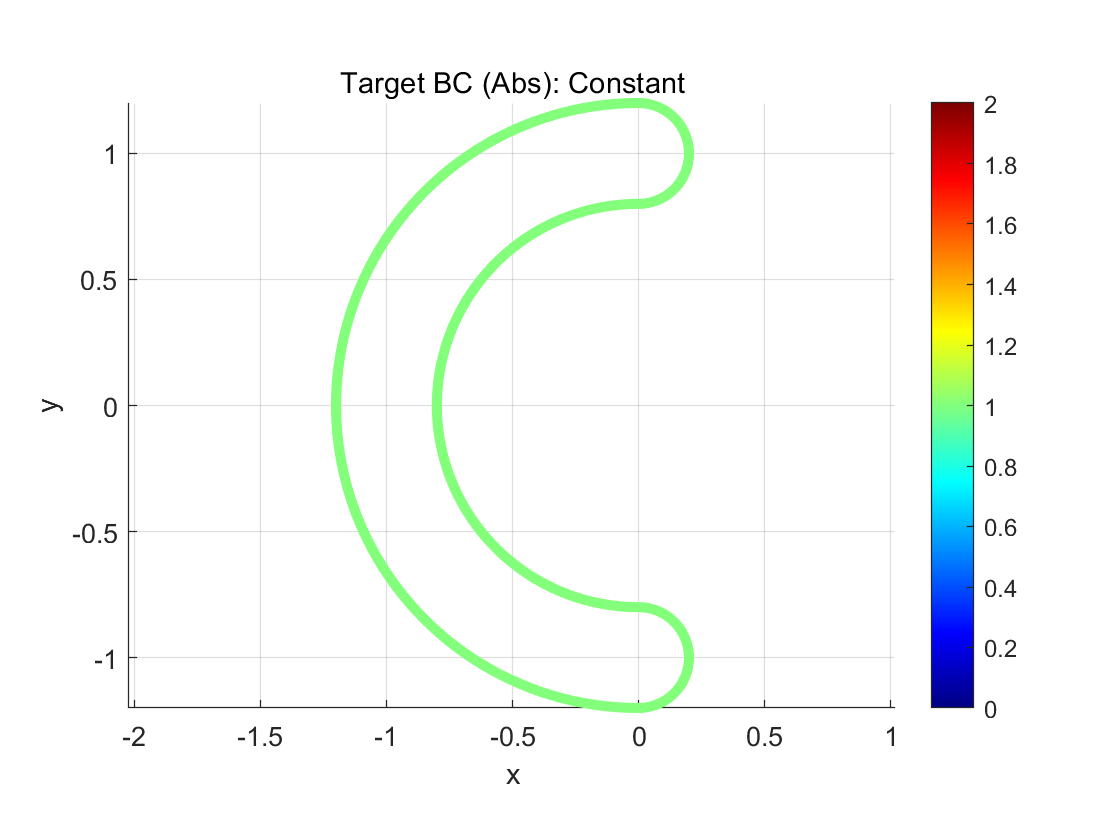}
        \caption{BC: Constant}
        \label{fig:bc_const}
    \end{subfigure}
    \caption{Visualization of the boundary conditions for Test 3. (a) The External Pulse BC simulates localized sources at the domain extremities. (b) The Constant BC requires a uniform value across the entire boundary.}
    \label{fig:multicenter_bcs}
\end{figure}

The regularization parameter $\alpha$ is selected via GCV for all cases.

\paragraph{Results}
The results are summarized in Table \ref{tab:multicenter_results} and the boundary error distributions are shown in Figure \ref{fig:multicenter_error}.

\begin{table}[H]
    \centering
    \caption{Comparison of Single-Center vs. Multi-Center BB-LbNM on the C-Shape ($M_{total}=600$).}
    \label{tab:multicenter_results}
    \begin{tabular}{l|l|c}
        \hline
        \textbf{BC Type} & \textbf{Method} & \textbf{$L^2$ Error}  \\
        \hline
        External Pulse & Single-Center & $3.27 \times 10^4$  \\
                  & Multi-Center  & $7.07 \times 10^{-2}$  \\
        \hline
        Constant  & Single-Center & $1.36 \times 10^4$  \\
                  & Multi-Center  & $4.49 \times 10^{-2}$  \\
        \hline
    \end{tabular}
\end{table}

\begin{figure}[htbp]
    \centering
    \begin{subfigure}[b]{0.48\textwidth}
        \centering
        \includegraphics[width=\textwidth]{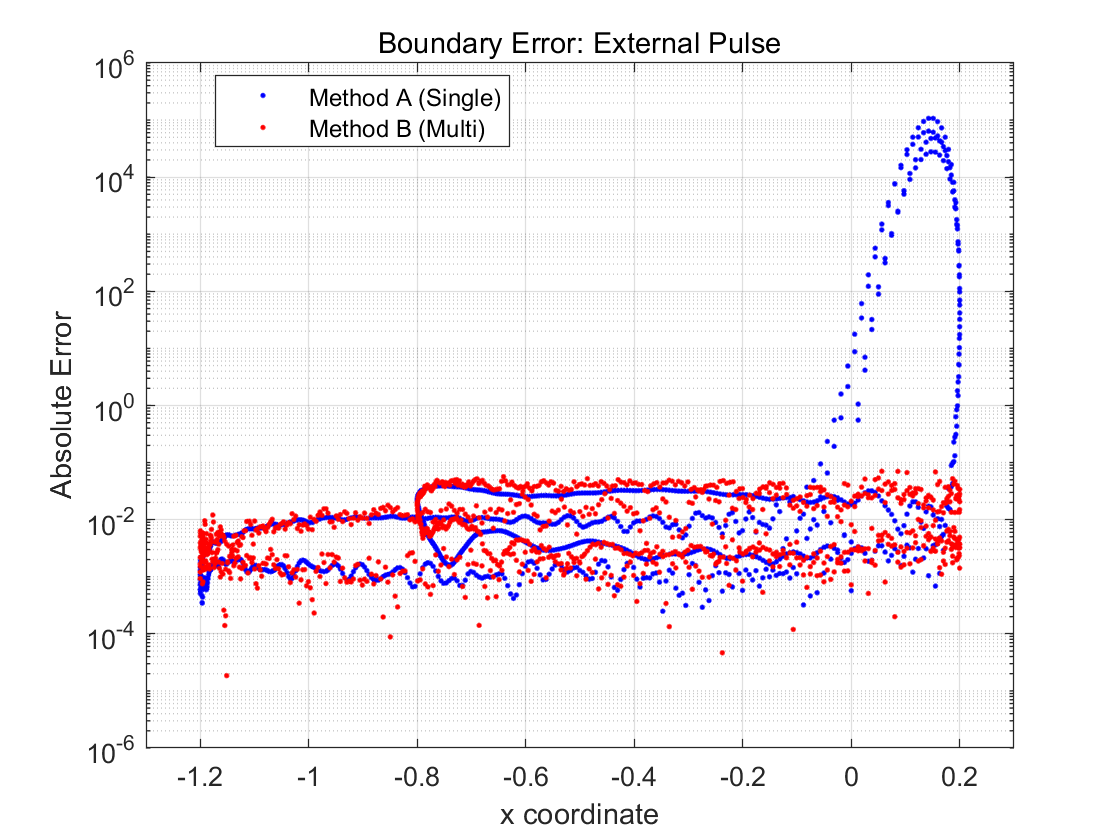}
        \caption{Error: Tip Pulse}
    \end{subfigure}
    \hfill
    \begin{subfigure}[b]{0.48\textwidth}
        \centering
        \includegraphics[width=\textwidth]{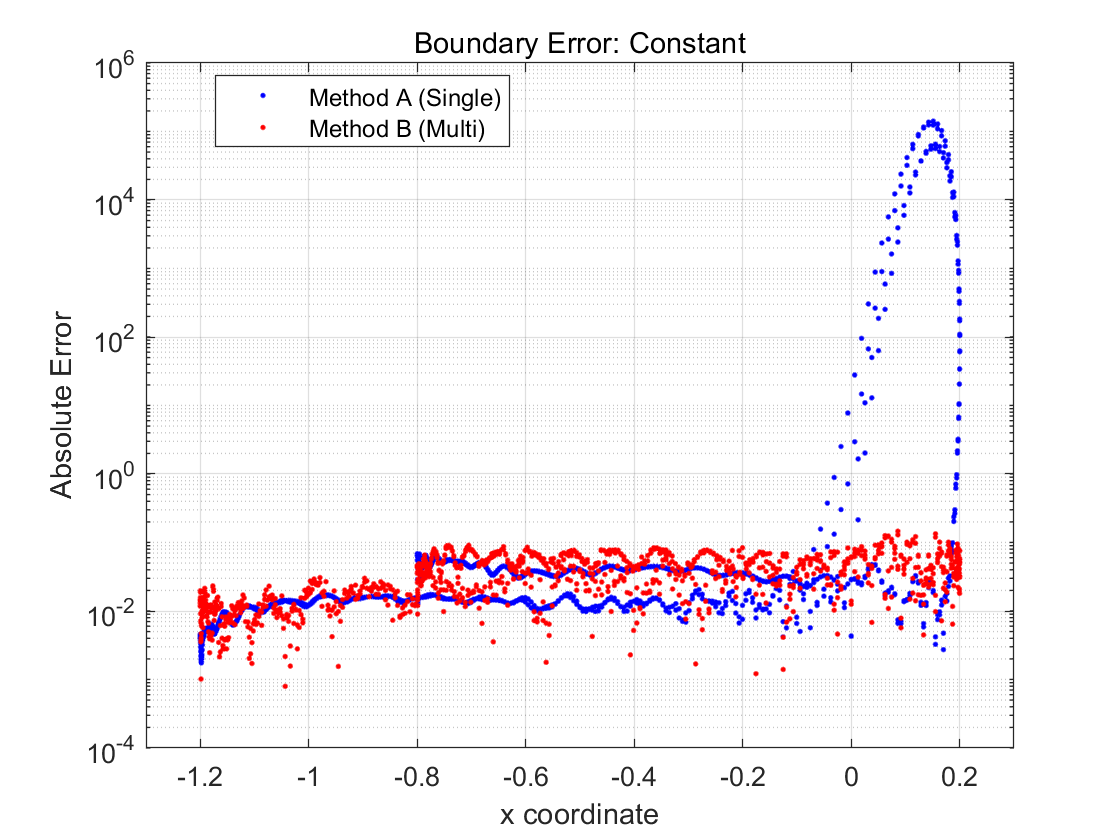}
        \caption{Error: Constant}
    \end{subfigure}
    \caption{Boundary error comparison. The Single-Center approach (Blue dots) exhibits divergence at the domain tips ($x>0$). The Multi-Center strategy (Red dots) successfully reconstructs the field across the entire boundary.}
    \label{fig:multicenter_error}
\end{figure}

The results demonstrate the practical necessity of the Multi-Center strategy for geometries with significant radial disparity. The Single-Center method yields errors of magnitude $10^4$. As visualized in Figure \ref{fig:multicenter_error}, the error remains low near the expansion center but explodes at the domain tips.

As analyzed previously, this failure results from the drastic magnitude difference of the basis functions between the near and far fields. The high-order terms required to resolve the tips ($r_{max}$) are numerically negligible at the back ($r_{min}$), making it impossible to simultaneously fit the boundary conditions at these disparate distances.

In contrast, the Multi-Center strategy reduces the error to $10^{-2}$. By distributing centers along the centerline, the method decouples the global problem into local neighborhoods where the radial dynamic range remains stable. This effectively bypasses the numerical scaling issue, ensuring robust approximation across the entire boundary.

\section{Conclusion}
\label{sec:conclusion}

In this work, we have addressed the critical computational challenges associated with solving the high-frequency Helmholtz equation in dissipative media and complex domains. We proposed and analyzed a learning-based numerical method using a Bessel basis (BB-LbNM), positioning it as a robust alternative to both the fundamental solution-based method (FS-LbNM) and the standard Finite Element Method (FEM).

The proposed BB-LbNM overcomes the fundamental limitations of existing approaches through its intrinsic properties. 
First, we identified that the FS-LbNM suffers from irreversible signal loss in dissipative media due to the exponential decay of its kernel, and from geometric mismatch errors on complex shapes due to its fixed source configuration. The Bessel basis, representing global standing waves, is robust to dissipation and naturally adaptable to arbitrary domains.
Second, we derived a rigorous convergence theory showing that the error is governed by the term $(d/\rho)^{-M_h}$. This rate depends solely on the intrinsic length scales of the problem, thereby eliminating the sensitive, user-defined geometric hyperparameters required by the FS approach.

Numerical experiments provided decisive validation of these claims:
\begin{enumerate}
    \item Stability in Dissipation: The BB-LbNM maintains machine-precision accuracy in highly dissipative regimes where the FS-LbNM fails due to numerical underflow.
    \item Geometric Robustness: The BB-LbNM successfully adapts to complex non-convex engineering geometries, without requiring manual geometric tuning.
    \item Efficiency versus FEM: By utilizing exact solutions of the governing equation, the BB-LbNM eliminates the pollution effect, achieving significantly higher accuracy than the FEM with a greatly reduced computational cost.
    \item Framework Extensibility: By employing a multi-center strategy, the method successfully overcomes the convergence limitations in non-star-shaped domains caused by radial disparities, demonstrating the framework's capability to handle topologically complex geometries.
\end{enumerate}

 Future work will focus on extending the method to three-dimensional problems using spherical Bessel functions, applying the framework to more general boundary conditions, and developing adaptive algorithms that rigorously analyze local singular behaviors (e.g., at corners or interfaces) to construct optimally enriched basis sets.

In conclusion, the Bessel basis offers a theoretically sound, numerically robust, and operationally simple foundation for learning-based Helmholtz solvers, establishing it as a superior choice for high-frequency engineering applications.


\section*{CRediT authorship contribution statement}

\textbf{Lifu Song:}  Methodology, Investigation, Formal analysis, Writing - original draft.
\textbf{Tingyue Li:} Methodology, Investigation, Formal analysis,  Writing - review \& editing. 
\textbf{Jin Cheng:} Supervision, Methodology, Funding acquisition, Formal analysis, Conceptualization, Writing - review \& editing.

\section*{Declaration of competing interest}

The authors declare that they have no known competing financial interests or personal relationships that could have appeared to influence the work reported in this paper.

\section*{Data availability}

No data were used for the research described in the article.

\section*{Acknowledgement}

This work is supported by National Key Research and Development Programs of China (Nos. 2024YFA1012401), National Natural Science Foundation of China (Nos. 12126601, 12201386, 12241103), Science and Technology Commission of Shanghai Municipality (23JC1400501), the Fundamental Research Funds for the Central Universities (No. 2025110603).

Thanks are due to Dr.Yu Chen from Shanghai University of Finance and Economics for valuable discussions.


\bibliography{references.bib}

@article{Chen2025,
  AUTHOR = {Chen, Yu and Cheng, Jin and Li, Tingyue and Miao, Yun},
     TITLE = {A learning based numerical method for {H}elmholtz equations
              with high frequency},
   JOURNAL = {J. Comput. Phys.},
  FJOURNAL = {Journal of Computational Physics},
    VOLUME = {520},
      YEAR = {2025},
     PAGES = {Paper No. 113478, 20},
      ISSN = {0021-9991,1090-2716},
   MRCLASS = {65N80 (35J05 65K10)},
  MRNUMBER = {4806464},
       DOI = {10.1016/j.jcp.2024.113478},
       URL = {https://doi.org/10.1016/j.jcp.2024.113478},
}

@article{RulandSalo2019,
  author = {Rüland, Angkana and Salo, Mikko},
    title = {Quantitative Runge Approximation and Inverse Problems},
    journal = {International Mathematics Research Notices},
    volume = {2019},
    number = {20},
    pages = {6216-6234},
    year = {2018},
    month = {01},
    issn = {1073-7928},
    doi = {10.1093/imrn/rnx301},
    url = {https://doi.org/10.1093/imrn/rnx301},
}

@book{abramowitz1964handbook,
  TITLE = {Handbook of mathematical functions, with formulas, graphs and
              mathematical tables},
    SERIES = {National Bureau of Standards Applied Mathematics Series},
    VOLUME = {No. 55},
    EDITOR = {Abramowitz, Milton and Stegun, Irene A.},
      NOTE = {Fifth printing, with corrections,
              National Bureau of Standards, Washington, D.C., (for sale by
              the Superintendent of Documents)},
 PUBLISHER = {U. S. Government Printing Office, Washington, DC},
      YEAR = {1966},
     PAGES = {xiv+1046},
   MRCLASS = {65.05 (00.20)},
  MRNUMBER = {208798},
}

@book{Evans2010pde,
  AUTHOR = {Evans, Lawrence C.},
     TITLE = {Partial differential equations},
    SERIES = {Graduate Studies in Mathematics},
    VOLUME = {19},
   EDITION = {Second},
 PUBLISHER = {American Mathematical Society, Providence, RI},
      YEAR = {2010},
     PAGES = {xxii+749},
      ISBN = {978-0-8218-4974-3},
   MRCLASS = {35-01},
  MRNUMBER = {2597943},
MRREVIEWER = {Diego\ M.\ Maldonado},
       DOI = {10.1090/gsm/019},
       URL = {https://doi.org/10.1090/gsm/019},
}

@book{Golub2013,
  author    = {Golub, Gene H. and Van Loan, Charles F.},
  title     = {Matrix Computations},
  edition   = {4th},
  year      = {2013},
  publisher = {Johns Hopkins University Press},
  address   = {Baltimore, MD, USA},
}

@article{Hansen2007,
  AUTHOR = {Hansen, Per Christian},
     TITLE = {Regularization {T}ools version 4.0 for {M}atlab 7.3},
   JOURNAL = {Numer. Algorithms},
  FJOURNAL = {Numerical Algorithms},
    VOLUME = {46},
      YEAR = {2007},
    NUMBER = {2},
     PAGES = {189--194},
      ISSN = {1017-1398,1572-9265},
   MRCLASS = {65F22},
  MRNUMBER = {2358249},
MRREVIEWER = {Rosemary\ A.\ Renaut},
       DOI = {10.1007/s11075-007-9136-9},
       URL = {https://doi.org/10.1007/s11075-007-9136-9},
}

@article{Lax1956Stability,
  AUTHOR = {Lax, P. D.},
     TITLE = {A stability theorem for solutions of abstract differential
              equations, and its application to the study of the local
              behavior of solutions of elliptic equations},
   JOURNAL = {Comm. Pure Appl. Math.},
  FJOURNAL = {Communications on Pure and Applied Mathematics},
    VOLUME = {9},
      YEAR = {1956},
     PAGES = {747--766},
      ISSN = {0010-3640,1097-0312},
   MRCLASS = {35.0X},
  MRNUMBER = {86991},
MRREVIEWER = {F.\ Browder},
       DOI = {10.1002/cpa.3160090407},
       URL = {https://doi.org/10.1002/cpa.3160090407},
}

@article{KuttlerSigillito1978Bounding,
   AUTHOR = {Kuttler, J. R. and Sigillito, V. G.},
     TITLE = {Bounding eigenvalues of elliptic operators},
   JOURNAL = {SIAM J. Math. Anal.},
  FJOURNAL = {SIAM Journal on Mathematical Analysis},
    VOLUME = {9},
      YEAR = {1978},
    NUMBER = {4},
     PAGES = {768--778},
      ISSN = {0036-1410},
   MRCLASS = {35P15 (49G05)},
  MRNUMBER = {492948},
MRREVIEWER = {L.\ E.\ Payne},
       DOI = {10.1137/0509056},
       URL = {https://doi.org/10.1137/0509056},
}

@article{Babuskasauter1997Pollution,
  AUTHOR = {Babu{\v{s}}ka, Ivo M. and Sauter, Stefan A.},
     TITLE = {Is the pollution effect of the {FEM} avoidable for the
              {H}elmholtz equation considering high wave numbers?},
      NOTE = {Reprint of SIAM J. Numer. Anal. {\bf 34} (1997), no. 6,
              2392--2423 [MR1480387 (99b:65135)]},
   JOURNAL = {SIAM Rev.},
  FJOURNAL = {SIAM Review},
    VOLUME = {42},
      YEAR = {2000},
    NUMBER = {3},
     PAGES = {451--484},
      ISSN = {1095-7200,0036-1445},
   MRCLASS = {65N30},
  MRNUMBER = {1786934},
       DOI = {10.1137/S0036142994269186},
       URL = {https://doi.org/10.1137/S0036142994269186},
}

@article{IhlenburgBabushka1997HP,
  author  = {Ihlenburg, F. and Babu{\v{s}}ka, I.},
     TITLE = {Finite element solution of the {H}elmholtz equation with high
              wave number. {II}. {T}he {$h$}-{$p$} version of the {FEM}},
   JOURNAL = {SIAM J. Numer. Anal.},
  FJOURNAL = {SIAM Journal on Numerical Analysis},
    VOLUME = {34},
      YEAR = {1997},
    NUMBER = {1},
     PAGES = {315--358},
      ISSN = {0036-1429},
   MRCLASS = {65N30},
  MRNUMBER = {1445739},
MRREVIEWER = {Manfred\ Dobrowolski},
       DOI = {10.1137/S0036142994272337},
       URL = {https://doi.org/10.1137/S0036142994272337},
}

@article{Ainsworth2004Dispersion,
  AUTHOR = {Ainsworth, Mark},
     TITLE = {Discrete dispersion relation for {$hp$}-version finite element
              approximation at high wave number},
   JOURNAL = {SIAM J. Numer. Anal.},
  FJOURNAL = {SIAM Journal on Numerical Analysis},
    VOLUME = {42},
      YEAR = {2004},
    NUMBER = {2},
     PAGES = {553--575},
      ISSN = {0036-1429,1095-7170},
   MRCLASS = {65N30},
  MRNUMBER = {2084226},
MRREVIEWER = {Srinivasan\ Kesavan},
       DOI = {10.1137/S0036142903423460},
       URL = {https://doi.org/10.1137/S0036142903423460},
}

@article{Nayfeh1975Aero,
  author = {Nayfeh, Ali H. and Kaiser, John E. and Telionis, Demetri P.},
    title = {Acoustics of Aircraft Engine-Duct Systems},
    journal = {AIAA Journal},
    volume = {13},
    number = {2},
    pages = {130-153},
    year = {1975},
    doi = {10.2514/3.49654},
}

@book{Carcione2014Seismic,
  title={Wave Fields in Real Media},
    editor = {José M. Carcione},
publisher = {Elsevier},
edition = {Third Edition},
address = {Oxford},
year = {2015},
isbn = {978-0-08-099999-9},
doi = {https://doi.org/10.1016/B978-0-08-099999-9.09991-X},
url = {https://www.sciencedirect.com/science/article/pii/B978008099999909991X}
}

@article{IsakovLu2018,
  AUTHOR = {Isakov, Victor and Lu, Shuai},
     TITLE = {Increasing stability in the inverse source problem with
              attenuation and many frequencies},
   JOURNAL = {SIAM J. Appl. Math.},
  FJOURNAL = {SIAM Journal on Applied Mathematics},
    VOLUME = {78},
      YEAR = {2018},
    NUMBER = {1},
     PAGES = {1--18},
      ISSN = {0036-1399,1095-712X},
   MRCLASS = {35R30 (35J05 78A46)},
  MRNUMBER = {3740331},
MRREVIEWER = {Paul\ Andrew\ Martin},
       DOI = {10.1137/17M1112704},
       URL = {https://doi.org/10.1137/17M1112704},
}

@article{Wang2024Stability,
  AUTHOR = {Wang, Tianjiao and Xu, Xiang and Yuan, Ganghua and Zhao, Yue},
     TITLE = {Stability for the inverse source problem of the {H}elmholtz
              equation with damping},
   JOURNAL = {SIAM J. Appl. Math.},
  FJOURNAL = {SIAM Journal on Applied Mathematics},
    VOLUME = {85},
      YEAR = {2025},
    NUMBER = {4},
     PAGES = {1438--1457},
      ISSN = {0036-1399,1095-712X},
   MRCLASS = {35R30 (35B35 78A46)},
  MRNUMBER = {4932250},
       DOI = {10.1137/24M1675035},
       URL = {https://doi.org/10.1137/24M1675035},
}

@book{Cobbold2006Ultrasound,
    author = {Cobbold, Richard S C},
    title = {Foundations of Biomedical Ultrasound},
    publisher = {Oxford University Press},
    year = {2006},
    month = {09},
    isbn = {9780195168310},
    doi = {10.1093/oso/9780195168310.001.0001},
    url = {https://doi.org/10.1093/oso/9780195168310.001.0001},
}

@article{Millar1973Rayleigh,
  AUTHOR = {Millar, R. F.},
     TITLE = {The {R}ayleigh hypothesis and a related least-squares solution
              to scattering problems for periodic surfaces and other
              scatterers},
   JOURNAL = {Radio Sci.},
  FJOURNAL = {Radio Science},
    VOLUME = {8},
      YEAR = {1973},
     PAGES = {785--796},
      ISSN = {0048-6604,1944-799X},
   MRCLASS = {78.41},
  MRNUMBER = {331989},
MRREVIEWER = {V.\ H.\ Weston},
       DOI = {10.1029/RS008i008p00785},
       URL = {https://doi.org/10.1029/RS008i008p00785},
}

@book{Hafner1990GMT,
  title={The Generalized Multipole Technique for Computational Electromagnetics},
  author={Hafner, Christian},
  year={1990},
  publisher={Artech House},
  address={Boston}
}

@book{Bergman1961Integral,
  AUTHOR = {Bergman, Stefan},
     TITLE = {Integral operators in the theory of linear partial
              differential equations},
    SERIES = {Ergebnisse der Mathematik und ihrer Grenzgebiete [Results in
              Mathematics and Related Areas]},
    VOLUME = {Band 23},
      NOTE = {Second revised printing},
 PUBLISHER = {Springer-Verlag New York, Inc., New York},
      YEAR = {1969},
     PAGES = {x+145},
   MRCLASS = {35.00},
  MRNUMBER = {239239},
}

@book{Gilbert1974Constructive,
   AUTHOR = {Gilbert, Robert P.},
     TITLE = {Constructive methods for elliptic equations},
    SERIES = {Lecture Notes in Mathematics},
    VOLUME = {365},
 PUBLISHER = {Springer-Verlag, Berlin-New York},
      YEAR = {1974},
     PAGES = {vii+397},
   MRCLASS = {35J15 (30A93 30A97 35CXX)},
  MRNUMBER = {447784},
MRREVIEWER = {Manfred\ W.\ Kracht},
}

@book{Vekua1967New,
  AUTHOR = {Vekua, I. N.},
     TITLE = {New methods for solving elliptic equations},
    SERIES = {North-Holland Series in Applied Mathematics and Mechanics},
    VOLUME = {1},
      NOTE = {Translated from the Russian by D. E. Brown},
 PUBLISHER = {North-Holland Publishing Co., Amsterdam; Interscience
              Publishers John Wiley \& Sons, Inc., New York},
      YEAR = {1967},
     PAGES = {xii+358},
   MRCLASS = {35.42},
  MRNUMBER = {212370},
}
\end{document}